\documentclass[11pt]{amsart}
\usepackage{amssymb,amsthm,amsmath,amstext,mathrsfs}
\usepackage{mathtools}
\usepackage{mathdots}
\usepackage{array}
\usepackage{varwidth}
\usepackage{xcolor}
\usepackage{multirow}
\usepackage{enumitem}
\usepackage{cases}
\usepackage{graphicx}

\usepackage[left=1.0in,top=1.05in,right=1.0in,bottom=0.95in]{geometry}

\usepackage[all]{xy}

\usepackage{tikz}
\tikzset{every picture/.style={line width=0.75pt}} 

\usepackage{float}

\newcommand{\Z}{\mathbb{Z}}

\newcommand{\F}{\mathbb{F}}
\newcommand{\PP}{\mathbb{P}}
\newcommand{\g}[2]{g^{#1}_{#2}}
\newcommand{\M}{\mathcal{M}}
\newcommand{\cal}[1]{\mathcal{#1}}
\newcommand{\rk}{\operatorname{rk}}
\newcommand{\BN}[3]{\M^{#2}_{#1,#3}}
\newcommand{\KK}[3]{\mathcal{K}^{#2}_{#1,#3}}
\newcommand{\abs}[1]{|#1|}
\newcommand{\maxk}{\kappa}
\newcommand{\floor}[1]{\left\lfloor #1 \right\rfloor}

\newcommand{\dmax}{d_{max}}

\newcommand{\calO}{\mathcal{O}}
\DeclareMathOperator{\Pic}{Pic}
\DeclareMathOperator{\Cliff}{Cliff}
\DeclareMathOperator{\gon}{gon}

\newcommand{\isom}{\cong}

\usepackage[backref=page]{hyperref}
\hypersetup{pdftitle={Brill--Noether loci in low genus},pdfauthor={Richard Haburcak}} 
\hypersetup{colorlinks=true,linkcolor=blue,anchorcolor=blue,citecolor=blue}

\usepackage[nameinlink]{cleveref}

\newtheoremstyle{thmbld}{\topsep}{\topsep}{}{}{\itshape}{}{0.5em}{}

\newtheorem{theorem}{Theorem}[section]
\newtheorem{lemma}[theorem]{Lemma}
\newtheorem{prop}[theorem]{Proposition}

\newtheorem{conj}{Conjecture}
\newtheorem{question}{Question}

\theoremstyle{definition}
\newtheorem{defn}[theorem]{Definition}

\theoremstyle{definition}
\newtheorem{remark}[theorem]{Remark}
\newtheorem{example}[theorem]{Example}

\newtheorem{innercustomthm}{Theorem}
\newenvironment{theoremintroalph}[1]
	{\renewcommand\theinnercustomthm{#1}\innercustomthm}
	{\endinnercustomthm}

\newtheorem{innercustomconj}{Conjecture}
\newenvironment{conjectureintroalph}[1]
	{\renewcommand\theinnercustomconj{#1}\innercustomconj}
	{\endinnercustomconj}

\newtheorem{innercustomquestion}{Question}
\newenvironment{questionintroalph}[1]
	{\renewcommand\theinnercustomquestion{#1}\innercustomquestion}
	{\endinnercustomquestion}

\newcolumntype{A}{w{c}{0.7cm}}

\title{Brill--Noether loci in genus $\leq 12$}

\author{Richard Haburcak}
\address{Department of Mathematics\\%
	The Ohio State University\\%
	100 Math Tower\\%
	Columbus, OH 43210}
\email{haburcak.1@osu.edu}

\begin{document}
	
	\thispagestyle{empty}
	
	\vspace*{-1.2cm}
	
	\begin{abstract}
		A refined Brill--Noether theory seeks to determine which linear series are admitted by a ``general'' curve in a particular Brill--Noether locus. However, as Brill--Noether loci are not irreducible in general, a coarse answer is given by the relative positions of Brill--Noether loci. Via an analysis of unstable Lazarsfeld--Mukai bundles on K3 surfaces, we distinguish Brill--Noether loci and provide expectations for the relative positions of Brill--Noether loci in general. Together with classical results, the refined Brill--Noether theory for curves of fixed gonality and on Hirzebruch surfaces, and explicit constructions, we identify the relative positions of all Brill--Noether loci in genus $g\leq 12$.
	\end{abstract}
	
	
	\maketitle
	
	\vspace*{-.8cm}
	
	\section*{Introduction}\label{sec: Intro}
	
	The study of special linear series on curves has a rich history. Classical results of Brill--Noether theory have focused on general algebraic curves. The Brill--Noether--Petri theorem~\cite{gieseker,griffiths_harris,lazarsfeld:Brill-Noether_without_degenerations} states that a general smooth curve of genus $g$ admits a linear series of degree $d$ and dimension $r$, called a $\g{r}{d}$, if and only if the \emph{Brill--Noether number} \[\rho(g,r,d)\coloneqq g-(r+1)(g-d+r)\] is non-negative. However, the curves we tend to encounter are special in moduli and classical results concerning special families of curves abound, for example Max Noether's theorem on the gonality of smooth plane curves and more recent generalizations~\cite{Coppens_91,Copopens_Kato_gon_nodal_plane_curves,Coppens_Kato_nontriv_plane_94}, Clifford's theorem on hyperelliptic curves and results for small Clifford index~\cite{Martens_Cliff_1}, linear series of curves on special varieties, e.g. scrolls and ruled surfaces~\cite{Accola79_castelnuovo_curves,ACGH,Castelnuovo1889}, and more recent results for curves of low gonality~\cite{Ballico_Keem_96,Coppens_Martens_99,Coppens_Martens_4gonal,Maroni_46_trigonal,Martens_96}.
	
	In the last few years, there has been a renewed focus on a \emph{refined Brill--Noether theory}, which aims to characterize the linear series on special curves, for example the ``general" curve in a fixed \emph{Brill--Noether locus} \[\BN{g}{r}{d}\coloneqq \{C\in \M_g \ : \ C \text{ admits a } \g{r}{d}\}\] when $\rho(g,r,d)<0$. A refined Brill--Noether theory has been recently established for curves of fixed gonality (addressing the $r=1$ case) by work of many authors in~\cite{ahl_2024,cook-powell_jensen,jensen_ranganathan,larson_larson_vogt_2020global,larson_refined_BN_Hurwitz,pflueger} (see also the excellent survey~\cite{jensen_payne_2021recent}, and~\cite{Coppens_Martens_4gonal,Coppens_Martens_4_gonal_vample,Larson_trigonal} for low gonality results), and for smooth curves on Hirzebruch surfaces by Larson--Vemulapalli in~\cite{larson_2024brillnoethertheorysmoothcurves}.
	
	Brill--Noether loci are, however, not irreducible in general (as observed in~\cite[appendix~A]{pflueger_legos} and~\cite[Corollary~4.7, Theorem~5.7]{h_tib_theta}); moreover, the curves parameterized by different components can exhibit very different behavior, as shown in~\cite{h_tib_theta}. Thus we aim to characterize the relative positions of Brill--Noether loci, to provide a coarse answer. By adding base points and subtracting non-base points, we obtain the \emph{trivial containments} $\BN{g}{r}{d}\subset \BN{g}{r}{d+1}$ and $\BN{g}{r}{d}\subset \BN{g}{r-1}{d-1}$. However, non-trivial containments and non-containments remain mysterious. The refined Brill--Noether theory for curves of fixed gonality~\cite{ahl_2024}, and curves on K3 surfaces~\cite{auel_haburcak_2022,ahk_2024,Farkas_2001,Lelli_Chiesa_2013} have furnished many non-containments of Brill--Noether loci. Recently, curves on K3 surfaces were used to show that the \emph{expected maximal Brill--Noether loci}, those not admitting any trivial containments, in fact admit no containments, except in genus $7,8$, and $9$~\cite[Theorem~1]{ahk_2024}. 
	
	Our main result is an identification of the relative positions of Brill--Noether loci in low genus.
	
	\begin{theoremintroalph}{A}\label{thmintro_main_thm}
		The relative positions of Brill--Noether loci in genus $g \leq 6$ are given by trivial containments. For $7 \leq g\leq 12$, the relative positions are identified in \Cref{fig:g_7_cont} (genus $7$), \Cref{fig:g_8_cont} (genus $8$), \Cref{fig:g_9_cont} (genus $9$), \Cref{fig:g_10_cont} (genus $10$), \Cref{fig:g_11_cont} (genus $11$), and \Cref{fig:g_12_cont} (genus $12$).
	\end{theoremintroalph}
	
	Additionally, we provide expectations for the relative positions of Brill--Noether loci in general via an analysis of unstable Lazarsfeld--Mukai bundles on K3 surfaces, and highlight additional techniques for distinguishing Brill--Noether loci, which may reduce the problem of identifying the relative positions to a small list where explicit constructions are required to prove containments.
	
	In his proof of the Brill--Noether--Petri theorem in~\cite{lazarsfeld:Brill-Noether_without_degenerations}, Lazarsfeld showed that a Brill--Noether special curve on a K3 surface degenerates to a reducible curve, stemming from the non-simplicity of the Lazarsfeld--Mukai bundle. By fixing the Picard group of the K3 surface, the degenerations of $C$ and the destabilizing subsheaves of the Lazarsfeld--Mukai bundle can be constrained. This is the approach taken in~\cite{ahk_2024}, where K3 surfaces with a fixed Picard group are considered, which lie in a divisor $\KK{g}{r}{d}$ of the moduli space $\mathcal{K}_g$ of quasi-polarized K3 surfaces of genus $g$. Though in~\cite{ahk_2024}, the range of $g,r,d$ was restricted so that $C$ degenerates to a reducible curve in an essentially unique way on the surface, giving strong numerical conditions that restrict the existence of a $\g{s}{e}$.
	
	When $C$ admits many degenerations to reducible curves, the analysis becomes more nuanced. In effect, the Lazarsfeld--Mukai bundle of a $\g{s}{e}$ on $C$ may be unstable in many ways, though there are still many numerical restrictions that can be summarized as a Gelfand--Tsetlin pattern, see \Cref{subsec: GT patterns}. By considering the stable factors of the Jordan--H\"{o}lder filtrations of the Harder--Narasimhan factors of the unstable bundle, one arrives at a \emph{terminal filtration} of the unstable bundle, from which bounds on the second Chern class can be obtained, which has led to progress on a conjecture of Donagi--Morrison on extending line bundles from curves to K3 surfaces~\cite{auel_haburcak_2022,Haburcak_BNspecial_k3_2024,Lelli_Chiesa_2013,Lelli_Chiesa_2015}, and restricting the existence of a $\g{s}{e}$ on $C$, which has been useful for distinguishing Brill--Noether loci. 
	
	We extract the philosophy that K3 surfaces predict the relative positions of Brill--Noether loci, witnessing the behavior of ``general" curves in a Brill--Noether locus. Namely, we expect K3 surfaces with fixed Picard group $\Z[H]\oplus \Z[L]$, where $L$ is a lift of a $\g{r}{d}$ on smooth irreducible curves $C\in|H|$, to capture the Brill--Noether theory of $C\in\BN{g}{r}{d}$. In retrospect, K3 surfaces also predict the refined Brill--Noether theory for curves of fixed gonality, by recent work of Farkas--Feyzbakhsh--Rojas in~\cite{farkas_feyz_Rojas_2025} using Bridgeland stability techniques to give a proof of the refined Brill--Noether theory for curves of fixed gonality on K3 surfaces, in the case that $L$ is a lift of a $\g{1}{k}$. Accordingly, we say that a containment $\BN{g}{r}{d}\subset \BN{g}{s}{e}$ is \emph{K3-expected} if a smooth irreducible curve $C\in|H|$ on such a K3 surface admits a $\g{s}{e}$, and expect that the K3-expected containments are indeed containments of Brill--Noether loci. Via the theory of Lazarsfeld--Mukai bundles, this can be made more precise, as $C$ admits a $\g{s}{e}$ exactly when the K3 surface admits its Lazarsfeld--Mukai bundle, which may sometimes be constructed as a sum of stable factors, obtained by considering all possible terminal filtrations of such a bundle. On the other hand, terminal filtrations impose lower bounds on $e$, giving non-containments of Brill--Noether loci. Thus, when the terminal filtrations can be understood, K3 surfaces give a heuristic for the relative positions of Brill--Noether loci.
	
	Inspired by the behavior of curves on K3 surfaces, we obtain expectations for the relative positions of Brill--Noether loci in general. However, as the terminal filtrations quickly become complicated when $g,r,d$ correspond to Brill--Noether loci with very negative $\rho$, we expect these trends to hold in some range where $g$ and $d$ are sufficiently large, see \Cref{subsec:non_conts_no_H-L_subline} for further discussion on the range of $g$ and $d$.
	
	\begin{conjectureintroalph}{A}[{\Cref{conj:r_leq_s_gen_conj}}]
		For $g$ and $d$ sufficiently large, $d,e\leq g-1$, and $2\leq r<s$
		\begin{itemize}
			\item if $e< d-2r+s+\frac{g-d+r+1}{2}+\frac{(s-2)(r-1)-1}{s-1}$, then $\BN{g}{r}{d}\nsubseteq\BN{g}{s}{e}$,
			\item if $e\geq d-2r+s+\frac{g-d+r+1}{2}+\frac{(s-2)(r-1)-1}{s-1}$, then $\BN{g}{r}{d}\subset \BN{g}{s}{e}$.
		\end{itemize}
	\end{conjectureintroalph}
	
	Another source of non-trivial containments of Brill--Noether loci arises by projecting from singular points, or from highly secant hyperplanes, to the image of a curve under a $\g{r}{d}$. More precisely, we consider the determinantal cycle \[V^{\ell+1}_{k}(\g{r}{d})\coloneqq \{D\in C_k \ \vert \ \dim \g{r}{d} (-D)\geq r-\ell-1\},\] of effective divisors of degree $k$ imposing at most $\ell+1$ independent conditions on $\g{r}{d}$. When $V^{r-s}_{d-e}(\g{r}{d})$ is non-empty, one obtains a $\g{s}{e}$. The expected dimension of $V^{r-s}_{d-e}(\g{r}{d})$ is $r-s-(d-e-r+s)s$, and any component has at least the expected dimension. When the expected dimension is non-negative, one obtains \emph{secant expected} containments, which hold when the expected dimension is positive. Conversely, one may expect that ``general" curves with a $\g{r}{d}$ admit no $\g{s}{e}$ when the expected dimension is negative. This is not the case, however, as even when the expected dimension of $V^{r-s}_{d-e}(\g{r}{d})$ is negative there may be non-secant expected containments of Brill--Noether loci, as evidenced by the containments arising from the refined Brill--Noether theory for fixed gonality, see also \Cref{prop:Vogt_g3_10_construction} and \Cref{rmk:k3exp_secunex} where the containments are K3-expected. Perhaps when the expected dimension of $V^{r-s}_{d-e}(\g{r}{d})$ is not too negative there may in fact be a containment $\BN{g}{r}{d}\subset \BN{g}{s}{e}$. Nevertheless, one might expect that secant expected containments are the ``first" source of non-trivial containments among Brill--Noether loci, at least in some range of $g,r,d,s,e$.
	
	\begin{conjectureintroalph}{B}[{\Cref{conj:secant_noncont_s_eq_2}}]
		For $g$ sufficiently large, if $r\geq 3$, $d,e\leq g-1$, and $e<d-2r+\floor{\frac{r+3}{2}}$, then $\BN{g}{r}{d}\nsubseteq \BN{g}{2}{e}$.
	\end{conjectureintroalph}
	
	Using the analysis of curves on K3 surfaces supports this conjecture, and work of Lelli-Chiesa in~\cite{Lelli_Chiesa_2013} verifies this in the range $e<d-2r+5$, under some mild technical restrictions on $g,r,d$ to ensure nice K3 surfaces exist, in particular avoiding $(-2)$-curves. For additional discussion, and see \Cref{subsec:conts_secants}.
	
	We also raise some questions regarding the irreducibility of Brill--Noether loci. Results on the irreducibility of Brill--Noether loci are sparse. It is well-known that loci with $r=1$ are irreducible, being the images of irreducible Hurwitz spaces, and it was shown in~\cite{CHOI2022,h_tib_theta} that loci with $r=2$ and $d$ sufficiently large are irreducible, the unique component being the image of the irreducible Severi variety of plane curves. Furthermore, the existence of a component of $\BN{g}{r}{d}$ of expected dimension ($3g-3+\rho(g,r,d)$) is also known only when $\rho$ is not too negative by independent work of Pflueger in~\cite{pflueger_legos} and Teixidor i Bigas in~\cite{bigas2023brillnoether}. As observed in~\cite[Appendix~A]{pflueger_legos}, when $\BN{g}{r}{d}$ has a component of expected dimension, and contains a gonality locus of larger dimension, one obtains a second component. In~\cite{h_tib_theta}, this was used to find optimal bounds on the irreducibility of $\BN{g}{2}{d}$. Hence perhaps the non-trivial containments of gonality loci are the ``first" origin of reducible Brill--Noether loci. The range of $g,r,d$ where one might expect Brill--Noether loci to be irreducible can be made precise using the refined Brill--Noether theory for fixed gonality, see \Cref{subsec: background refined BN gonality} and \Cref{subsec:irred_conjs} for more detail. Let $\kappa(g,r,d)\coloneqq \max\{k \ \vert \ \BN{g}{1}{k}\subset \BN{g}{r}{d}\}$, as in~\cite{ahl_2024} where a closed formula for $\kappa(g,r,d)$ is given, see also \Cref{subsec: background refined BN gonality}. 
	
	\begin{questionintroalph}{A}[{\Cref{ques:irred_kappa_bound}}]
		Is $\BN{g}{r}{d}$ irreducible when $0>\rho(g,r,d)>\rho(g,1,\kappa(g,r,d))$?
	\end{questionintroalph}

	\subsection*{Acknowledgments}
	The author would like to thank David Anderson, Asher Auel, Andrei Bud, Hannah Larson, Carl Lian, Montserrat Teixidor i Bigas, and Sameera Vemulapalli for helpful conversations, and particularly to Dave Jensen and Isabel Vogt for discussions and insight, and Andreas Leopold Knutsen for helpful comments on an earlier draft of this paper. This research was partially conducted during the time that R.H. was supported by the National Science Foundation under Grant No. DMS-2231565.
	
	\subsection*{Outline}
	In \Cref{sec: Background}, we recall background on Brill--Noether loci in \Cref{subsec: background BN loci}, special linear series and classical results in \Cref{subsec: classical (non)cont BN}, relevant results from the refined Brill--Noether theory for curves of fixed gonality in \Cref{subsec: background refined BN gonality}, Lazarsfeld--Mukai bundles and stability in \Cref{subsec: background LM stability}, and curves on K3 surfaces in \Cref{subsec: background_curves_k3}. In \Cref{sec: Adm_assign_GT}, we define admissible assignments for a terminal filtration of an unstable Lazarsfeld--Mukai bundle, give bounds on admissible assignments in \Cref{subsec: Decomp}, and summarize how they are used for distinguishing Brill--Noether loci in \Cref{subsec: Dist BN via GF}. In \Cref{sec: BN_strat_3_10}, we begin the proof of \Cref{thmintro_main_thm}, for genus $3$--$6$ in \Cref{subsec: BNStrat genus 3_6}, genus $7$ in \Cref{subsec: BN Strat genus 7}, genus $8$ in \Cref{subsec: BN Strat genus 8}, genus $9$ in \Cref{subsec: BN Strat genus 9}, and genus $10$ in \Cref{subsec: BN Strat genus 10}. In \Cref{sec:BN_strat_ge_11} we give some general containments and non-containments in higher genus, and return to the proof of \Cref{thmintro_main_thm} in genus $11$ and $12$ in \Cref{subsec: BN Strat genus 11} and \Cref{subsec:BN Strat genus 12}. Finally, in \Cref{sec:Conjectures}, we motivate and state questions regarding the reducibility and relative positions of Brill--Noether loci in general.

	\section{Background}\label{sec: Background}
	
	We recall some background on Brill--Noether loci, some classical constructions giving containments and non-containments of Brill--Noether loci, and Lazarsfeld--Mukai bundles on K3 surfaces. In particular, we recall some useful results to distinguish Brill--Noether loci.
	
	Recall that the \emph{Clifford index} of a line bundle $A\in\Pic(C)$ is defined as \[\Cliff(A)\coloneqq \deg (A)-2(h^0(C,A)-1),\] and the \emph{Clifford index of a curve} is defined as \[\Cliff{C}\coloneqq \min \{\Cliff{A} \ \vert \ h^0(C,A),h^1(C,A)\geq 2\},\] and a line bundle is said to \emph{compute the Clifford index} of $\Cliff(A)=\Cliff(C)$.

	\subsection{Brill--Noether loci}\label{subsec: background BN loci}
	
	The Brill--Noether--Petri theorem implies that when $\rho(g,r,d)<0$, the Brill--Noether locus $\BN{g}{r}{d}$ is a proper subvariety of $\M_g$, and we assume $g,r,d$ satsify $\rho(g,r,d)<0$ when considering Brill--Noether loci. We note that Serre duality gives $\BN{g}{r}{d}=\BN{g}{g-d+r-1}{2g-2-d}$, thus we freely assume $d\leq g-1$ throughout. 
	
	The geometry of Brill--Noether loci remains mysterious in general; even basic questions, such as the dimension, are unknown. The expected dimension of $\BN{g}{r}{d}$ is $3g-3+\rho(g,r,d)$, any component having at least the expected dimension, as shown in~\cite{steffen_1998}. Additionally, the existence of a component of expected dimension is known only when $\rho$ is not too negative, see for example~\cite{bh_2024maximal,Knutsen_Lelli_Chiesa_Mongardi_BN_abelian_surf,pflueger_legos,Sernesi_1984,bigas2023brillnoether}. Less is known about equidimensionality of components, which is known in the range $-3\leq\rho\leq-1$, assuming additionally that $g\geq 12$ when $\rho=-3$~\cite{EdidinThesis,steffen_1998}. 
	
	Similarly, little is known about irreducibility of Brill--Noether loci in general, though loci with $\rho=-1$ are known to be irreducible by work of Eisenbud--Harris in~\cite{Eisenbud_Harris_1989}, and loci with $r=2$ and $d$ sufficiently large are known to be irreducible by work of Teixidor i Bigas and the author~\cite[Proposition~4.2]{h_tib_theta}, where sharp bounds are obtained. Furthermore, Brill--Noether loci may admit components of larger than expected dimension arising from curves of constant gonality (see~\cite{h_tib_theta} and~\cite[Appendix~A]{pflueger_legos}) or Castelnuovo curves (see~\cite[Remark~1.4]{pflueger_legos}). Complicating the picture further, there may be multiple components of the expected dimension, as shown in~\cite{h_tib_theta}, where for $r\geq 3$ and $g=\binom{r+2}{2}$, it is shown that $\BN{g}{r}{g-1}$ admits (at least) two components of expected dimension, one where the general curve does not admit a theta characteristic of dimension $r$, and the other where the $\g{r}{g-1}$ is a theta characteristic.
	
	A refined Brill--Noether theory may instead seek to characterize the linear series for a general curve in each component of a Brill--Noether locus. However, as the number of components and a description of each component of Brill--Noether loci is currently unknown, for the time being we content ourselves by understanding the relative positions of Brill--Noether loci, seeking to describe the relative position of components when possible. We note that using chains of loops or chains of elliptic curves in~\cite{pflueger_legos,bigas2023brillnoether}, there is a useful description of some components when $\rho$ is not too negative.
	
	There are many containments among Brill--Noether loci. For example, Clifford's theorem implies the containments $\BN{g}{r}{2r}\subset \BN{g}{1}{2}$. There are also the \emph{trivial containments} $\BN{g}{r}{d}\subseteq \BN{g}{r}{d+1}$ obtained by adding a base point; and $\BN{g}{r}{d}\subseteq \BN{g}{r-1}{d-1}$ when $r\geq 2$ obtained by subtracting a non-base point, as observed in~\cite{Farkas2000,Lelli-Chiesa_the_gieseker_petri_divisor_g_le_13}. The refined Brill--Noether theory for curves of fixed gonality~\cite{pflueger,jensen_ranganathan,larson_refined_BN_Hurwitz} and curves on Hirzebruch surfaces~\cite{larson_2024brillnoethertheorysmoothcurves} provide many containments and non-containments of Brill--Noether loci for $r=1$ and $r=2$, see \Cref{subsec: background refined BN gonality} for more detail on the former. Furthermore, non-containments of Brill--Noether loci, proven using limit linear series techniques, have been useful int he study of the Kodaira dimension of $\M_g$ in work of Eisenbud--Harris in~\cite{eisenbud_harris} and Farkas in~\cite{Farkas2000}. 
	
	Since the work of Lazarsfeld in~\cite{lazarsfeld:Brill-Noether_without_degenerations}, K3 surfaces have been a staple of Brill--Noether theory. Indeed, curves on K3 surfaces with prescribed Picard group have proven useful in distinguishing Brill--Noether loci. In~\cite{Farkas_2001}, Farkas studies the gonality of curves on K3 surfaces in $\PP^3$, distinguishing certain Brill--Noether loci and giving a new proof of the fact that the Kodaira dimension of $\M_{23}$ is at least $2$. Work of Lelli--Chiesa in~\cite{Lelli_Chiesa_2013} on the Donagi--Morrison conjecture, which asks when Brill--Noether special linear series on a curve lift to the K3 surface, has yielded non-containments of Brill--Noether loci of the form $\BN{g}{r}{d}\nsubseteq \BN{g}{1}{e}$ and $\BN{g}{r}{d}\nsubseteq \BN{g}{2}{e}$ for a range of $g,r,d,e$, via an analysis of destabilizing subsheaves of Lazarsfeld--Mukai bundles. However, similar non-containments $\BN{g}{r}{d}\nsubseteq\BN{g}{s}{e}$ for $s\geq 3$ appear significantly more complicated, though some are obtained in~\cite{auel_haburcak_2022} for $s=3$. Inspired by the known non-containments, in~\cite{auel_haburcak_2022}, Auel and the author defined the \emph{expected maximal Brill--Noether loci} as those admitting no trivial containments, and conjectured that the expected maximal Brill--Noether loci should in fact admit no containments, except for the known exceptions in genus $7,8,9$. Limit linear series techniques and non-containments coming from the refined Brill--Noether theory for curves of fixed gonality in~\cite{ahl_2024,bh_2024maximal,bigas2023brillnoether} provided partial progress, before the conjecture was resolved in~\cite{ahk_2024} via an analysis of Lazarsfeld--Mukai bundles on K3 surfaces.

	\subsection{Classical (non-)containments of (components of) Brill--Noether loci}\label{subsec: classical (non)cont BN}
	
	We recall a few classical results on linear series on curves. We begin with classical results of Castelnuovo and Severi restricting the genus of curves.

	\begin{theorem}[Castelnuovo's bound]\label[theorem]{thm:castelnuovo_bound}
		If $C$ admits a basepoint free and birationally very ample $\g{r}{d}$, then letting \[m=\floor{\frac{d-1}{r-1}} \text{ and }\epsilon = d-1-m(r-1),\] we have \[g(C)\leq \frac{m(m-1)(r-1)}{2}+m\epsilon.\]
	\end{theorem}
	
	\begin{theorem}[Castelnuovo--Severi inequality]\label[theorem]{thm:castelnuovo_severi}
		Suppose $f_1:C\to C_1$ and $f_2: C\to C_2$ are two non-constant maps to curves of genus $g_1$ and $g_2$, respectively, of degree $d_1$ and $d_2$, respectively. If $f_1$ and $f_2$ do not factor through a common map $\tilde{f}:C\to \widetilde{C}$, then \[g(C)\leq (d_1 -1)(d_2 -1) + d_1 g_1 + d_2 g_2.\]
	\end{theorem}
	
	We now give containments of Brill--Noether loci for curves with a divisor of Clifford index $1$.
	
	\begin{lemma}\label[lemma]{lem:Cliff_1_cont}
		Let $g\geq 7$, $r\geq 2$, and $d\leq g-1$ be positive integers. If $d-2r=1$, then $\BN{g}{r}{d}=\BN{g}{1}{2}$. 
	\end{lemma}
	\begin{proof}
		Let $C\in \BN{g}{r}{d}$. We aim to show that $\Cliff(C)=0$, whereby Clifford's Theorem shows that $C\in \BN{g}{1}{2}$ and we have $\BN{g}{r}{d}\subseteq \BN{g}{1}{2}$. The result then follows from the fact that $\BN{g}{1}{2}$ is contained in every Brill--Noether locus.
		
		Let $D$ be an effective divisor on $C$ of type $\g{r}{d}$. As $\Cliff(D)=d-2r=1$, $\Cliff(C)\leq 1$. 
		Suppose for contradiction that $\Cliff(C)=1$. Martens shows in~\cite[2.56]{Martens_Cliff_1} that for an effective divisor $D$ on $C$, with $\deg(D)\leq g-1$, and $\Cliff(D)=\Cliff(C)$ odd, we have $h^0(C,D)<\Cliff(C)+2$, except if $g=3(\Cliff(C)+1)$. As $g\geq 7$, we must have $3\leq r+1=h^0(C,D)<3$, which is a contradiction. Thus $\Cliff(C)=0$, as desired.
	\end{proof}
	
	Unfortunately, such characterizations of Brill--Noether loci for linear series of higher Clifford index are not so simple.

	In~\cite{Lange_moduli_curves_rat_maps}, Lange shows that the moduli space $M(\gamma,k)$ of curves with a degree $k$ map to a curve of geometric genus $\gamma \geq 1$ is $\dim M(\gamma,k)=2g-2 -(2k-3)(\gamma-1)$. In particular, if $C\in \BN{g}{r}{d}$, from the determinantal nature of $\BN{g}{r}{d}$, we must have \[3g-3+\rho(g,r,d)\leq \dim \BN{g}{r}{d}\leq \dim M(\gamma ,k) = 2g-2-(2k-1)(\gamma -1),\] which we restate below for reference.
	
	\begin{theorem}[{\cite{Lange_moduli_curves_rat_maps}}]\label[theorem]{thm:lange_birat_dim}
		Let $C\in \BN{g}{r}{d}$ and suppose $C$ admits a degree $k$ map to a curve of geometric genus $\gamma\geq 1$, then \[3g-3 +\rho(g,r,d) \leq 2g-2 -(2k-3)(\gamma-1).\]
	\end{theorem}
	
	Gonality will be a useful way to distinguish (components of) Brill--Noether loci, and we recall a few bounds on the gonality for different families of curves.
	
	For a smooth plane curve of degree $d$, Max Noether's theorem states that the gonality is $d-1$ and is obtained via projection from a point on the curve. For a curve with a singular plane model, projection from a node gives a $\g{1}{d-2}$. This gives a simple containment of Brill--Noether loci, which we summarize here for later reference.
	
	\begin{lemma}\label[lemma]{lem:proj_from_nodes_plane}
		Let $g,d$ be positive integers with $d\geq 4$ and $\rho(g,2,d)<0$. If $g<\frac{(d-1)(d-2)}{2}$, then $\BN{g}{2}{d}\subseteq \BN{g}{1}{d-2}$.
	\end{lemma}
	\begin{proof}
		For a curve $C\in \BN{g}{2}{d}$ either the $\g{2}{d}$ is not very ample, in which case subtracting points and trivial containments show that $C$ admits a $\g{1}{d-2}$; or the $\g{2}{d}$ is very ample, and the image of $C$ in $\PP^2$ is a curve with arithmetic genus $\frac{(d-1)(d-1)}{2}>g$, thus the image has a point of multiplicity $\geq 2$, and projecting from this point shows that $C$ has a $\g{1}{k}$ for some $k\leq d-2$, hence also a $\g{1}{d-2}$.
	\end{proof}
	
	We note that there are cases where $C$ may admit a $\g{1}{3}$, as in~\cite[Example~4.1]{Copopens_Kato_gon_nodal_plane_curves}, though we do not use this here. Building on work in~\cite{Copopens_Kato_gon_nodal_plane_curves}, Coppens proves in~\cite{Coppens_91} the following result, which shows that $\BN{g}{2}{d}\nsubseteq \BN{g}{1}{d-3}$ in many cases.
	
	\begin{theorem}[{\cite{Coppens_91}}]\label[theorem]{thm:coppens_plane_gon}
		Let $C$ be the normalization of a general integral nodal plane curve $\Gamma$ of degree $d$ with $\delta=\binom{d-1}{2} -g$ nodes. If $\rho(g,1,d-3)<0$, then $\gon(C)=d-2$.
	\end{theorem}
	
	We recall a lower bound on the gonality of a smooth complete intersection due to Lazarsfeld.
	
	\begin{theorem}[{\cite[Exer.~4.12]{Laz_lect_lin_ser}}]\label[theorem]{thm:gon_comp_int}
		Let $C\subset \PP^r$ be a smooth complete intersection of type $(a_1,\dots,a_{r-1})$ where $2\leq a_1 \leq a_2 \leq \cdots \leq a_{r-1}$. Then \[\gon(C)\geq (a_1 -1)a_2 \cdots a_{r-1}.\]
	\end{theorem}
	
	Finally, we recall how highly secant hyperplanes, or more generally effective divisors imposing independent conditions, give additional linear series, see~\cite{Farkas_2008}. Suppose $C\subset \PP^r$ is a curve of genus $g$ and degree $d$. If $C$ has a $k$-secant $\ell$-plane, then projection away from the $\ell$-plane gives a $\g{r-\ell-1}{d-k}$ on $C$. More generally, given $l=(L,V)\in G^r_d(C)$, and $0\leq k <d$, one may consider the determinantal cycle \[V^{\ell+1}_{k}(l)\coloneqq \{D\in C_{k} \ \mid \ \operatorname{dim}l(-D)\geq r-\ell-1\}\] of effective divisors of degree $k$ that impose at most $\ell+1$ independent conditions on $l$. If the $\g{r}{d}$ is an embedding, then $V^{\ell+1}_{k}(l)$ parameterizes $k$-secant $\ell$-planes to the image of $C\subset \PP^r$, and projection from such a hyperplane gives a $\g{r-\ell-1}{d-k}$. More generally, when $V^{r-s}_{d-e}(\g{r}{d})\neq \emptyset$, one obtains a $\g{s}{e}$. 
	
	The expected dimension of $V^{\ell+1}_{k}(l)$ is \[\operatorname{exp}\dim V^{\ell+1}_{k}(l) = k-(k-\ell-1)(r-\ell),\] and any component of $V^{\ell+1}_{k}(l)$ has dimension at least the expected dimension. 
	\begin{defn}
		When $\operatorname{exp}\dim V^{r-s}_{d-e}\geq0$, we call the containment $\BN{g}{r}{d} \subseteq \BN{g}{s}{e}$ \emph{secant expected}. Conversely, when $\operatorname{exp}\dim V^{r-s}_{d-e}<0$, the containment is called \emph{non-secant expected}. 
	\end{defn}

	The question of non-emptiness of $V^{\ell+1}_{k}$ is, however, difficult in general. While the virtual classes of the cycles have been computed, see for example~\cite[Chapter~VIII]{ACGH}, the formulas are unwieldy; more manageable results exist in special cases, as in~\cite{Castelnuovo1889}, computing the virtual number of $(2r-2)$-secant $(r-2)$-planes to curves in $\PP^r$. However, there are few results on the cycles $V^{\ell+1}{k}(l)$ for non-general curves. For further discussion, we direct the reader to~\cite{Farkas_2008} where the cycles $V^{\ell+1}_{k}(l)$ are studied for general curves.

	\subsection{Refined Brill--Noether for fixed gonality}\label{subsec: background refined BN gonality}
	
	As the Hurwitz space of degree $k$ covers is irreducible it makes sense to talk about the general $k$-gonal curve, the Brill--Noether locus $\BN{g}{1}{k}$ is the closure of the locus of $k$-gonal curves, and so is irreducible, and the general curve in $\BN{g}{1}{k}$ is general of gonality $k$.
	
	While $W^r_d(C)$ can have many components of varying dimension, work of Cook-Powell, Jensen, Larson, Larson, Pflueger, Ranganathan, Vogt and others~\cite{cook-powell_jensen,jensen_ranganathan,larson_refined_BN_Hurwitz,larson_larson_vogt_2020global} shows that for a general $k$-gonal curve, somponents of $W^r_d(C)$ are given by splitting type loci. Pflueger showed in~\cite{pflueger} that for $r'=\min\{r,g-d+r-1\}$, \[\dim W^r_d(C)\leq \rho_k(g,r,d)\coloneqq \rho(g,r,d)+ \max_{\ell\in \{0,\dots,r'\}}(g-k-d+2r+1)\ell -\ell^2, \] and conjectured that the bound was attained, which was later verified by Jensen--Ranganathan.
	
	\begin{theorem}[{\cite[Theorem~A]{jensen_ranganathan}}]\label[theorem]{thm:rho_k}
		If $C$ is a general curve of genus $g$ and gonality $k\geq2$, and assume that $g-d+r>0$, then $\dim W^r_d(C)=\rho_k(g,r,d)$. In particular, $C$ admits a $\g{r}{d}$ if and only if $\rho_k(g,r,d)\geq 0$.
	\end{theorem}
	
	As we freely assume $d\leq g-1$, we have $r'=r$. As in~\cite{ahl_2024}, let \[\kappa(g,r,d)\coloneqq \max\{k \ \mid \ \BN{g}{1}{k}\subseteq \BN{g}{r}{d}\}=\max \{k \ \mid \ \rho_k(g,r,d)\geq 0\}.\] One can use this to show non-containments of Brill--Noether loci.
	
	\begin{prop}[{\cite[Proposition~2.2]{ahl_2024}}]\label[prop]{prop:maxk_distinguish}
		If $\maxk(g,r,d)>\maxk(g,s,e)$, then $\BN{g}{r}{d}\nsubseteq \BN{g}{s}{e}$
	\end{prop}
	
	Furthermore, $\maxk(g,r,d)$ has a simple closed form, obtained in~\cite[Proposition~2.5]{ahl_2024}, for $d\leq g-1$, \begin{equation}
		\maxk(g,r,d)=\begin{cases}
	\floor{\frac{d}{r}} & \text{ if } g+1 > \floor{\frac{d}{r}}+d \\
	g+1-d+2r+\floor{-2\sqrt{-\rho(g,r,d)}} & \text{ else. }
	\end{cases}\label{eq:maxk_formula} \end{equation}

	\subsection{Lazarsfeld--Mukai bundles and stability on K3 surfaces}\label{subsec: background LM stability}
	We briefly recall the definition and basic properties of Lazarsfeld--Mukai bundles, and the notion of stability on K3 surfaces. For proofs and further details, we refer the interested reader to~\cite{aprodu,lazarsfeld:Brill-Noether_without_degenerations} and for stability to~\cite{huybrechts_lehn}.
	
	Let S be a K3 surface and $C\subset S$ a smooth irreducible curve of genus $g$. A basepoint free $\g{r}{d}$ on $C$, denoted by $(A,V)$ with $A\in\Pic^d(C)$ and $V\subseteq H^0(C,A)$ an $(r+1)$-dimensional subspace, gives rise to the Lazarsfeld--Mukai bundle $E_{C,(A,V)}$ on $S$, defined as the dual of the kernel of the evaluation map $V\otimes \cal{O}_S \twoheadrightarrow A$. We recall the following well-known properties of $E_{C,(A,V)}$:
	
	\begin{itemize}
		\item $\rk(E_{C,(A,V)})=r+1$, $c_1(E_{C,(A,V)}) = [C]$, $c_2(E_{C,(A,V)})=d$;
		\item $\chi(E_{C,(A,V)})=g-d+2r+1$, $h^1(E_{C,(A,V)})=h^0(C,A)-r-1$, $h^2(E_{C,(A,V)})=0$;
		\item $E_{C,(A,V)}$ is globally generated off a finite set;
		\item $\chi(E_{C,(A,V)}^\vee \otimes E_{C,(A,V)})=2(1-\rho(g,r,d))$. In particular, if $\rho(g,r,d)<0$, then $E_{C,(A,V)}$ is non-simple.
	\end{itemize}
	
	Being a Lazarsfeld--Mukai bundle is an open condition. That is, a vector bundle $E$ with $h^1(S,E)=h^2(S,E)=0$ such that $c_1(E)$ is represented by a smooth irreducible curve $C$ is the Lazarsfeld--Mukai bundle of a linear series on $C$. More generally, as defined in~\cite{Lelli_Chiesa_2015}, a torsion free coherent sheaf $E$ on $S$ is called a \emph{generalized Lazarsfeld--Mukai} bundle if $h^2(S,E)=0$, and either 
	\begin{itemize}
		\item[(I)] $E$ is locally free and generated by global sections of a finite set, or 
		\item[(II)] $E$ is globally generated.
	\end{itemize} 
	Furthermore, if both (I) and (II) are satisfied, then $E$ is the Lazarsfeld--Mukai bundle associated to a smooth irreducible curve $C$ and a primitive linear series $(A,V)$, with $V=H^0(C,A)$ when $h^1(S,E)=0$, as observed in~\cite[Remark~1]{Lelli_Chiesa_2015}. This will be useful when constructing Lazarsfeld--Mukai bundles as direct sums of line bundles and smaller Lazarsfeld--Mukai bundles.
	
	We now recall briefly stability of coherent sheaves on K3 surfaces, specifically Lazarsfeld--Mukai bundles.
	
	Let $H$ be an ample line bundle and $E$ be a torsion free sheaf on $S$. The slope of $E$ (with respect to $H$) is defined as \[\mu (E)\coloneqq \frac{c_1(E).H}{\rk(E)}.\] We call $E$ \emph{(slope)-stable} (resp., \emph{(slope)-semistable}) if for any coherent subsheaf $0\neq F\subset E$ with $\rk(F)<\rk(E)$, one has $\mu(F)< E$ (resp., $\mu(F)\leq \mu(E)$). We recall that a simple sheaf is stable.
	
	The Harder--Narasimhan filtration (HN filtration) is the unique filtration \[0=HN(E)_0 \subset HN(E)_1 \subset \cdots \subset HN(E)_k = E,\] with $HN(E)_{i+1} / HN(E)_{i}$ a torsion free semistable sheaf, and \[\mu(HN(E)_i/HN(E)_{i-1})>\mu(HN(E)_{i+1} / HN(E)_{i}).\] Moreover, if $E$ is a vector bundle, then the sheaves $HN(E)_i$ are locally free and \[\mu(HN(E)_1)>\mu(HN(E)_2)>\cdots > \mu(E).\]
	
	For a semistable sheaf $E$, the Jordan--H\"{o}lder filtration (JH filtration) is a filtration \[0=JH(E)_0 \subset JH(E)_1 \subset \cdots JH(E)_k=E,\] such that $JH(E)_{i+1}/JH(E)_i$ are torsion free stable sheaves with $\mu(JH(E)_{i+1}/JH(E)_i)=\mu(E)$. While the JH filtration always exists, it is not unique. 
	
	Finally, we recall that the moduli space of stable coherent sheaves of rank $r$ with fixed Chern classes $c_1$ and $c_2$ has dimension \[(1-r)c_1^2 + 2rc_2 -2r^2+2,\] hence if $E$ is stable, then \[c_2(E) \geq \frac{(\rk(E)-1)c_1(E)^2}{2\rk(E)} + \rk(E) - \frac{1}{\rk(E)}.\]
	
	\subsection{Curves on K3 surfaces}\label{subsec: background_curves_k3}
	
	We will consider curves on quasi-polarized K3 surfaces $(S,H)$ of genus $g$ with $\Pic(S)=\Lambda^r_{g,d}$, where $\Lambda^r_{g,d}$ is the lattice $\Z[H]\oplus \Z[L]$ with intersection matrix \[\begin{bmatrix}
	H^2 & H.L \\ H.L & L^2
	\end{bmatrix} = \begin{bmatrix}
	2g-2 & d \\ d & 2r-2
	\end{bmatrix}.\] 
	
	Denoting by $\mathcal{K}_g$ the moduli space of quasi-polarized K3 surfaces, pairs $(S,H)$ of a K3 surface $S$ of genus $g$ and $H\in\Pic(S)$ a big and nef line bundle with $H^2=2g-2$. From Hodge theory, the Noether--Lefschetz locus of K3 surfaces with Picard rank $>1$ is a countable union of irreducible divisors $\KK{g}{r}{d}$. For $g\geq 2$, $r\geq 0$, and $d\geq 0$, we denote by $\KK{g}{r}{d}$ the quasi-polarized K3 surfaces $(S,H)\in\mathcal{K}_g$ admitting a primitive embedding of $\Lambda^r_{g,d}$ preserving $H$, following notation from~\cite{auel_haburcak_2022,ahk_2024}. When $\Delta(g,r,d)\coloneqq 4(g-1)(r-1)-d^2<0$, such K3 surfaces exist by the surjectivity of the period map, see~\cite[Theorem~2.9(i)]{Morrison_k3} or~\cite{Nikulin_79}. Furthermore, as in~\cite[Section~1.2]{ahk_2024}, we may assume that $H$ is nef after acting by Picard--Lefschetz reflections. In fact, when $\Delta(g,r,d)<0$, $\KK{g}{r}{d}$ is a non-empty irreducible divisor (see~\cite{OGrady_irreducible_NL_divisors}), and the very general $(S,H)\in\KK{g}{r}{d}$ has $\Pic(S)=\Lambda^r_{g,d}$.
	
	In particular, for $g\geq 3$, $r\geq 1$, and $2\leq d\leq g-1$, smooth irreducible curves $C\in|H|$ have genus $g$ and carry a basepoint free complete $\g{r}{d}$, namely $|\calO_C(L)|$, see~\cite[Lemma~1.7]{ahk_2024}.

	\section{Admissible assignments}\label{sec: Adm_assign_GT}
	
	We use curves on K3 surfaces to distinguish Brill--Noether loci. Specifically, we consider K3 surfaces with $\Pic(S)=\Lambda^r_{g,d}$ and the Lazarsfeld--Mukai bundle $E_{C,\g{s}{e}}$ associated to a linear series $\g{s}{e}$ on $C\in|H|$ with $\rho(g,s,e)<0$. As $\Pic(S)$ is fixed, the determinants of the stable Jordan--H\"{o}lder subfactors of the Harder--Narasimhan subquotients of $E_{C,\g{s}{e}}$ are constrained, and lead to many restrictions on $\g{s}{e}$.
	
	\subsection{Terminal filtrations}\label{subsec: TermFilt}	
	Throughout this section, let $(S,H)$ be a polarized K3 surface of genus $g$, and $E$ be an unstable sheaf on $S$ with $c_1(E)=H$. Given $E$, it is natural to look at the HN filtration of $E$ and the JH filtrations of the HN factors. By lifting each JH filtration, one obtains the \emph{terminal filtration} of $E$, a filtration where each factor is stable and the slopes are ``decreasing". Here we set notation and recall the notions of terminal filtrations, and recall how a terminal filtration gives an expression for $c_2(E)$.
	
	To arrive at a terminal filtration for $E$, begin with the HN filtration of $E$
	\[0=HN(E)_0 \subset HN(E)_1 \subset \cdots \subset HN(E)_k = E,\] and the JH filtrations of each subquotient, $HN(E)_i/HN(E)_{i-1}$,
	\[0 \subset JH(HN(E_i)/HN(E_{i-1}))_1\subset \cdots \subset JH(HN(E_i)/HN(E_{i-1}))_{k_i-1}\subset HN(E_i)/HN(E_{i-1}).\] The subsheaves of the JH filtrations of each subquotient give subsheaves \[HN(E_{i-1})\subset\overline{JH(HN(E_i)/HN(E_{i-1}))_1}\subset\cdots\subset\overline{JH(HN(E_i)/HN(E_{i-1}))_{k_i-1}}\subset HN(E_i),\] which all together give a filtration of $E$
	\[0=E_0 \subset E_1 \subset \cdots \subset E_n= E,\] where $E_i/E_{i-1}$ is a torsion free stable sheaf, and the slopes satisfy $\mu(E_i/E_{i-1})\geq \mu(E_{i+1}/E_i)$, and if $E$ is a vector bundle then the slopes also satisfy $\mu(E_1)\geq \mu(E_2)\geq\cdots \geq \mu(E)$. 
	
	\begin{defn}
		A \emph{terminal filtration of length $n$ and type $r_1\subset\cdots\subset r_n$} of a sheaf $E$ is a filtration \[0=E_0\subset E_1 \subset \cdots \subset E_n=E\] of coherent sheaves with $\rk(E_i)=r_i$ such that $E_i/E_{i-1}$ is a torsion free stable sheaf and the slopes satisfy
		\[\mu(E_j/E_i)\geq \mu(E_k/E_i)\geq \mu(E_k/E_j)\] for ${0\leq i<j<k\leq n}$.
		We say that a terminal filtration is \emph{quotient non-negative} if ${c_1(E/E_i)^2\geq0}$ for all $0< i< n$, and is \emph{quotient slope-positive} if $\mu(E/E_i)>0$ for all $0<i<n$.
	\end{defn} 
	
		\begin{remark}
			Given a terminal filtration, one obtains an expression for $c_2(E)$ by recursively using the exact sequences \[0 \to E_{i-1}\to E_i\to E_i/E_{i-1}\to 0,\] which gives \[c_2(E_i)=c_2(E_{i-1})+c_2(E_i/E_{i-1})+c_1(E_{i-1}).c_1(E_i/E_{i-1}).\] In total, we obtain \[c_2(E)=\sum_{i=1}^{n} c_2(E_{i}/E_{i-1})+c_1(E_{i}/E_{i-1}).\left(\sum_{j=0}^{i-1}c_1(E_{j}/E_{j-1})\right).\] Since $E_i/E_{i-1}$ is stable for $1\leq i\leq n$, we have \[c_2(E_i/E_{i-1})\geq (r_i-r_{i-1})-\frac{1}{r_i-r_{i-1}}+\frac{r_i-r_{i-1}-1}{2(r_i-r_{i-1})}c_1(E_{i}/E_{i-1})^2.\] 
		\end{remark}
	
	\subsection{Admissible assignments are Gelfand--Tsetlin patterns}\label{subsec: GT patterns}
	
	Recall that a \emph{Gelfand--Tsetlin pattern (of size $n$)} is a triangular array of entries $x_{i,j}$ with $1\leq j\leq i \leq n$
	\[\xymatrix@!C=14pt@!R=9pt{
		x_{1,1} & & x_{2,1} & & x_{3,1} & & \cdots & & x_{n,1} \\
		& x_{2,2} & & x_{3,2} & & \cdots & & x_{n,2}\\
		& & x_{3,3} & & \cdots & & x_{n,3}\\
		& & & \mathbin{\rotatebox[origin=c]{-50}{$\cdots$}} & & \mathbin{\rotatebox[origin=c]{50}{$\cdots$}} & \\
		& & & & x_{n,n}
	}\] such that $x_{i,j}\geq x_{i+1, j+1} \geq x_{i+1, j}$.

	\begin{remark}
		The data of a terminal filtration of length $n$ can be arranged into a triangle of size $n$ with $x_{i,j}=E_{i}/E_{i-j}$ 
		\[\xymatrix@!C=14pt@!R=9pt{
			E_1 & & E_2/E_1 & & E_3/E_2 & & \cdots & & E_{n-2}/E_{n-3} & & E_{n-1}/E_{n-2} & & E_n/E_{n-1}\\
			& E_2 & & E_3/E_1 & & E_4/E_2& & \cdots & &E_{n-1}/E_{n-3} & & E_n/E_{n-2} & \\
			& & E_3 & & E_4/E_1 & & \cdots & & \cdots& & E_n/E_{n-3} \\
			& & & \mathbin{\rotatebox[origin=c]{-50}{$\cdots$}} & & \mathbin{\rotatebox[origin=c]{-50}{$\cdots$}} & & \mathbin{\rotatebox[origin=c]{50}{$\cdots$}} & & \mathbin{\rotatebox[origin=c]{50}{$\cdots$}}\\
			& & & & E_{n-2} & & E_{n-2}/E_{1} & & E_n/E_{2} \\
			& & & & & E_{n-1} & & E_n/E_{1}\\
			& & & & & & E_n 
			}\]		
		where the leftmost diagonal is simply the terminal filtration, and each subsequent diagonal is the quotient terminal filtration by $E_i$. Since $E_1\subset \cdots E_n$ is a terminal filtration, we have \[\mu(x_{i,j})=\mu(E_i/E_{i-j})\geq \mu(x_{i+1,j+1})=\mu(E_{i+1}/E_{i-j}) \geq \mu(x_{i+1,j})=\mu(E_{i+1}/E_{i-j+1}),\] hence taking the slopes gives a (rational) Gelfand--Tsetlin pattern, and in fact the slope conditions of a terminal filtration are equivalent to the Gelfand--Tsetlin conditions. The top row contains the stable subquotients of the terminal filtration, and the last North-East diagonal contains the quotients $E_n/E_i$.
	\end{remark}
	
	For a fixed terminal filtration type (say $r_1\subset \cdots \subset r_n$), as the ranks are known, the slope data is completely determined by the first Chern classes of $E_i$, or equivalently by $c_1(E_n/E_i)$, which will be slightly more convenient so work with when $E_n$ is a Lazarsfeld--Mukai bundle.
	
	\begin{defn}
		Given a sequence of integers $0=r_0<r_1<r_2<\cdots<r_n$, we call a collection of line bundles $L_1,L_2,\dots,L_n$ a \emph{slope admissible assignment} if $x_{i,j}=\frac{H.(L_i-L_{i-j})}{r_i-r_{i-j}}$ is a (rational) Gelfand--Tsetlin pattern. We call $L_1,L_2,\dots,L_n$ \emph{quotient non-negative} if in addition $(L_n-L_i)^2\geq 0$ for all $0<i<n$, and \emph{quotient slope-positive} if $H.(L_n-L_i)>0$ for all $0<i<n$.
	\end{defn}

	\begin{lemma}\label[lemma]{lem: lm term filt is qnneg}
		Given a Lazarsfeld--Mukai bundle $E$ and a terminal filtration \[E_1\subset \cdots \subset E_n=E,\] the collection of line bundles $c_1(E_1), c_1(E_2),\dots,c_1(E_n)$ is a quotient non-negative quotient slope-positive admissible assignment.
	\end{lemma}
	\begin{proof}
		As the first Chern classes come from a terminal filtration, it is clear that they give a slope admissible assignment. It remains to show that $c_1(E/E_i)^2\geq 0$ and $\mu(E/E_i)>0$. Let $M_i=E/E_i$. Since $M_i$ is a quotient of $E$, it is globally generated off a finite set of points. Moreover, we have ${h^2(S,M_i)=0}$, thus $h^0(S,\det M_i)\geq 2$ by~\cite[Lemma~3.3]{Lelli_Chiesa_2013} as $M_i^{\vee\vee}$ is globally generated off a finite set and $\det M_i = \det(M_i^{\vee\vee})$. As in~\cite[Lemma~3.2]{Lelli_Chiesa_2013}, we see that $\det M_i$ is base point free and nontrivial, thus $\mu(\det M_i)>0$, $h^2(S,\det M_i)=0$, and $\det M_i$ is nef whereby $c_1(M_i)^2\geq 0$.
	\end{proof}
	
	\begin{remark}
		To abbreviate the conditions, by an \emph{admissible assignment} for any Lazarsfeld--Mukai bundle we will always mean a quotient non-negative and quotient slope-positive admissible assignment.
	\end{remark}
	
	\begin{remark}
		Given an unstable Lazarsfeld--Mukai bundle, one can consider all possible terminal filtration types, and consider quotient non-negative admissible assignments for each type. We note that the existence of such an assignment does not imply the existence of such a terminal filtration. Indeed, there are additional restrictions coming from the theory of (generalized) Lazarsfeld--Mukai bundles for a (quotient non-negative quotient slope-positive) admissible assignment to arise from a terminal filtration. We give two such examples; the first coming from restrictions on Donagi--Morrison lifts of linear series, and the other coming from restrictions on generalized Lazarsfeld--Mukai bundles arising as quotients of Lazarsfeld--Mukai bundles.
	\end{remark}
	
	The presence of elliptic curves gives admissible assignments which do not arise from a terminal filtration.
	
	\begin{example}[Non-terminal admissible assignment in genus $11$]\label[example]{ex:g_11_nonterm_adm_asaign}
		Let $S$ be a K3 surface with $\Pic(S)=\Lambda^2_{11,7}$. We note that $\Lambda^2_{11,7}$ has no $(-2)$-curves, and has $(0)$-curves: $n(H-2L)$ and $n(-H+5L)$.
		
		There are a few admissible assignments of $E=E_{C,\g{3}{e}}$ giving a bound on $c_2$ which is $\leq 10$ (if the assignment gave $c_2>10$, then it would not be of a $\g{3}{10}$). We note that $h^0(E_{C,\g{3}{e}})=18-e$.
		
		For a terminal filtration of type $1\subset 4$, there are two admissible assignments with $c_2(E)\leq 10$, which have $c_1(E_1)=H-L$ or $c_1(E_1)=2L$, respectively. The former does not occur, as else the $\g{3}{10}$ would be contained in the $\g{2}{7}$ by~\cite[Lemma~4.1]{Lelli_Chiesa_2015}. The latter does not occur as then $E/2L$ would be a generalized Lazarsfeld--Mukai bundle with $c_1^2=0$ and~\cite[Proposition~2.7]{Lelli_Chiesa_2015} shows $E/2L\isom \calO_S(\Sigma)^{\oplus 3}$ for an irreducible elliptic curve $\Sigma\subset S$, which is not stable. So this admissible assignments does not arise from a terminal filtration!
			 
		Similarly, there are two admissible assignments of type $1\subset 2 \subset 4$ giving a bound $c_2(E)\leq 10$, one of which does not arise from a terminal filtration. The admissible assignment $c_1(E_1)=L$ and $c_1(E_2)=2L$ does not occur, as again $E/(E_2)$ is a generalized Lazarsfeld--Mukai bundle with $c_1^2=0$, and~\cite[Proposition~2.7]{Lelli_Chiesa_2015} shows $E/(E_2)$ is not stable. So this assignment does not come from a terminal filtration!
		
		However, there is an admissible assignment with $c_1(E_1)=L$, $c_1(E_2)=H-L$. In fact, \[E=L\oplus (H-2L)\oplus E_{L, \g{1}{2}}\] is a Lazarsfeld--Mukai bundle of a $\g{3}{10}$ with the given terminal filtration, taking \[E_1=L,\ E_2=L\oplus (H-2L).\] Indeed, $E$ is locally free and globally generated, $c_1(E)=H$, $c_2(E)=10$, and \[h^0(E)=h^0(L)+h^0(H-2L)+h^0(E_{L,\g{1}{2}})= 3+2+3=8.\] From the last admissible assignment, we obtain curves on K3 surfaces with $\Pic(S)=\Lambda^2_{11,7}$ that have a $\g{3}{10}$, giving the K3-expected containment $\BN{11}{2}{7}\subset \BN{11}{3}{10}$. In fact, we show that $\BN{11}{2}{7}\subset \BN{11}{3}{10}$, see \Cref{prop:Vogt_g3_10_construction}.
	\end{example}
	
	\begin{example}[Non-terminal admissible assignment in genus $100$]\label[example]{ex_3L_destab}
		Let $S$ be a K3 surface with $\Pic(S)=\Lambda^9_{100,57}$. We note that $\Lambda^9_{100,57}$ has a $(-2)$-curve $-H+4L$, and an elliptic curve $H-3L$. Suppose a smooth irreducible $C\in|H|$ has a $\g{4}{d}$. The Lazarsfeld--Mukai bundle $E=E_{C,\g{4}{d}}$ has a terminal filtration, and computing all quotient non-negative admissible assignments for each terminal filtration type shows that there are two assignments with the smallest bound on $c_2(E)$. They both occur when $E$ has a terminal filtration type $1 \subset 5$, and the admissible assignments are
		\begin{itemize}
			\item $E_1 = H-L$, giving $c_2(E)\geq 50.75$, and
			\item $E_1 = 3L$, giving $c_2(E)\geq 30.75$.
		\end{itemize} 
		We show that the latter does not occur as an admissible assignment for a terminal filtration. If it were from a terminal filtration, then this would give $E$ as an extension \[0 \to 3L \to E \to E' \to 0,\] where $E'$ is a stable vector bundle of rank $4$ with $c_1(E')=H-3L$. However, as $(H-3L)^2=0$, and $H-3L$ is an elliptic curve, \cite[Proposition~2.7]{Lelli_Chiesa_2015} shows that $c_2(E')=0$ and $E'\isom \calO_S(\Sigma)^{\oplus 4}$ for an irreducible elliptic curve $\Sigma\subset S$, which is not stable. Therefore, such a terminal filtration in fact does not exist! However, this is not immediate simply from the bounds on admissible assignments. 
		
		The same example arises for all $\g{r}{d}$ with $2\leq r \leq 8$. We note that the elliptic curve $H-3L$ gives a $\g{1}{27}$ on $C$, and $3L \oplus (H-3L)$ is the Lazarsfeld--Mukai bundle of the $\g{1}{27}$ on $C$.
	\end{example}

	\subsection{Decompositions of polarization}\label{subsec: Decomp}
	
	Given a filtration \[0=E_0\subset E_1 \subset \cdots \subset E_n=E\]of $E$, one obtains many decompositions of $c_1(E)=c_1(E_i)+c_1(E/E_i)$. We will focus on the case when the filtration is a terminal filtration, whence \Cref{lem: lm term filt is qnneg} immediately gives the following conditions:
	\begin{eqnarray}
		\mu(E_i)\geq\mu(E)\geq\mu(E/E_i)>0,\\
		c_1(E/E_i)^2\geq 0.
	\end{eqnarray}
	
	Thus, on a polarized K3 surface $(S,H)$ with a fixed Picard group $\Pic(S)=\Lambda^r_{g,d}$, an unstable Lazarsfeld--Mukai bundle $E$ with $c_1(E)=H$ gives many decompositions of the polarization $H$, which are constrained in many ways. As the Picard group is fixed, we can write \[c_1(E/E_i)=x_i H - y_i L,\] and we obtain the following conditions on $(x_i, y_i)$,
	
	\begin{equation}
		\frac{(2g-2)(1-x_i) +dy_i}{r_i}\geq \frac{2g-2}{r_{n}}\geq \frac{(2g-2)x_i - dy_i}{r_n-r_i}  >0, \label{eq:term_fil_quot_slope_ineq}
	\end{equation}
	\begin{equation}
		(2g-2)x_i^2 -2d x_i y_i +(2r-2)y_i^2 \geq 0. \label{eq:term_filt_quot_nonneg_selfint}
	\end{equation}
	
	We obtain the following simple bounds for $x_i$ and $y_i$ in terms of $\Lambda^r_{g,d}$, similar to~\cite[Lemma~3.2]{ahk_2024}.
	
	\begin{lemma}\label[lemma]{lem:ineq_destab_fil}
		Let $(S,H)$ be a polarized K3 surface with $\Pic(S)=\Lambda^r_{g,d}=\Z[H]\oplus \Z[L]$, denote by $\Delta=\Delta(g,r,d)=4(g-1)(r-1)-d^2$, and let \[0\to M \to E \to E/M \to 0\] be an exact sequence of sheaves with $c_1(E)=H$, $c_1(E/M)^2\geq 0$, and \[\mu(M)\geq \mu(E)\geq\mu(E/M)>0.\] Suppose $c_1(E/M)=xH-yL$. If $r=1$, then \[x=0 \text{ and } \frac{-2(g-1)}{d}<y<0, \text{ or } x=1 \text{ and } 0<y<\frac{g-1}{d}.\] If $r\geq 2$, then either \[0<x\leq 1+\frac{d}{\sqrt{\abs{\Delta}}},\text{ and } 0<y\leq \frac{2(g-1)}{\sqrt{\abs{\Delta}}},\] or \[1-\frac{d}{\sqrt{\abs{\Delta}}} \leq x \leq 0 \text{ and } -\frac{2(g-1)}{\sqrt{\abs{\Delta}}} \leq y <0.\]
	\end{lemma}
	\begin{remark}
		We note that in any case, we have \[\abs{x}\leq 1+\frac{d}{\sqrt{\abs{\Delta}}},\,\,\,\,\,   \abs{y}\leq \frac{2g-2}{\sqrt{\abs{\Delta}}}.\]
	\end{remark}
	\begin{proof}
		We first observe that as $c_1(E/M)=xH-yL$, we have
		\begin{equation}
			(g-1)x^2-dxy+(r-1)y^2\geq 0,	\,\,\,\,\,\left(\frac{1}{2}c_1(E/M)^2 \geq 0\right)	\label{eq:c_1(E/M)^2>=0}
		\end{equation}
		\begin{equation}
			(2g-2)x-dy>0, \,\,\,\,\,\left(\mu(E/M)>0\right)	\label{eq:mu(E/M)>0}
		\end{equation}
		\begin{equation}
			(2g-2)(1-x)+dy>0, \,\,\,\,\,\left(\mu(M)>0\right)	\label{eq:mu(M)>0}
		\end{equation}
		
		We first treat the case $r=1$.
		
		If $x<0$, then \eqref{eq:c_1(E/M)^2>=0} and \eqref{eq:mu(E/M)>0} yield \[(g-1)x \leq dy < 2(g-1)x,\] whence the contradiction	$0<(g-1)x$. If $x>0$, then \eqref{eq:c_1(E/M)^2>=0} and \eqref{eq:mu(M)>0} yield \[2(g-1)(x-1)< dy \leq (g-1)x,\] which shows $x=1$, whereby \eqref{eq:mu(M)>0} and \eqref{eq:c_1(E/M)^2>=0} show $0<dy\leq g-1$. When $x=0$, \eqref{eq:mu(E/M)>0} and \eqref{eq:mu(M)>0} yield $\frac{-2(g-1)}{d}<y<0$.
		
		We now treat the case that $r\geq 2$.
		
		We note that \eqref{eq:mu(E/M)>0} and \eqref{eq:mu(M)>0} together yield
		\begin{equation}
			\frac{2(g-1)}{d}x>y>\frac{2(g-1)}{d}(x-1)	\label{eq:mu(M),mu(E/M)>0},
		\end{equation}
		which shows that 
		\begin{equation}
			\text{either } x\leq 0 \text{ and } y<0,\, \text{ or } x>0 \,\,(\text{hence } x\geq 1)\,\, \text{ and } y>0 \label{eq:xparity_implies_yparity}.
		\end{equation}
		Let $a_{\pm}\coloneqq \frac{d\pm\sqrt{\abs{\Delta}}}{2(r-1)}$, and \eqref{eq:c_1(E/M)^2>=0} factors as \begin{equation}
			(r-1)(y-a_{+}x)(y-a_{-}x)\geq 0		\label{eq:factored_c1(E/M)^2>=0} .
		\end{equation}
		Define the following lines in the $(x,y)$-plane: \[\ell_{+}: y=a_{+}x,\,\,\,\,\,\,\,\,\, \ell_{-}: y= a_{-}x,\,\,\,\,\,\,\,\, \ell: y=\frac{2(g-1)}{d}(x-1). \]
		
		Since $a_{+}>\frac{2g-2}{d}>a_{-}>0$, the conditions \eqref{eq:mu(M),mu(E/M)>0}, \eqref{eq:factored_c1(E/M)^2>=0}, and \eqref{eq:xparity_implies_yparity} yield that either 
		\begin{equation}
			x>0 \text{ and } 0\leq\frac{2(g-1)}{d}(x-1)<y\leq a_{-}x	\label{eq: xybound_x>0}
		\end{equation}
		or
		\begin{equation}
			x\leq 0 \text{ and } \frac{2(g-1)}{d}(x-1)<y\leq a_{+}x,\,\,\, y<0	\label{eq: xybound_x<=0}.
		\end{equation}
		
		\begin{figure}[H]
			\centering
			\includegraphics[width=0.9\linewidth]{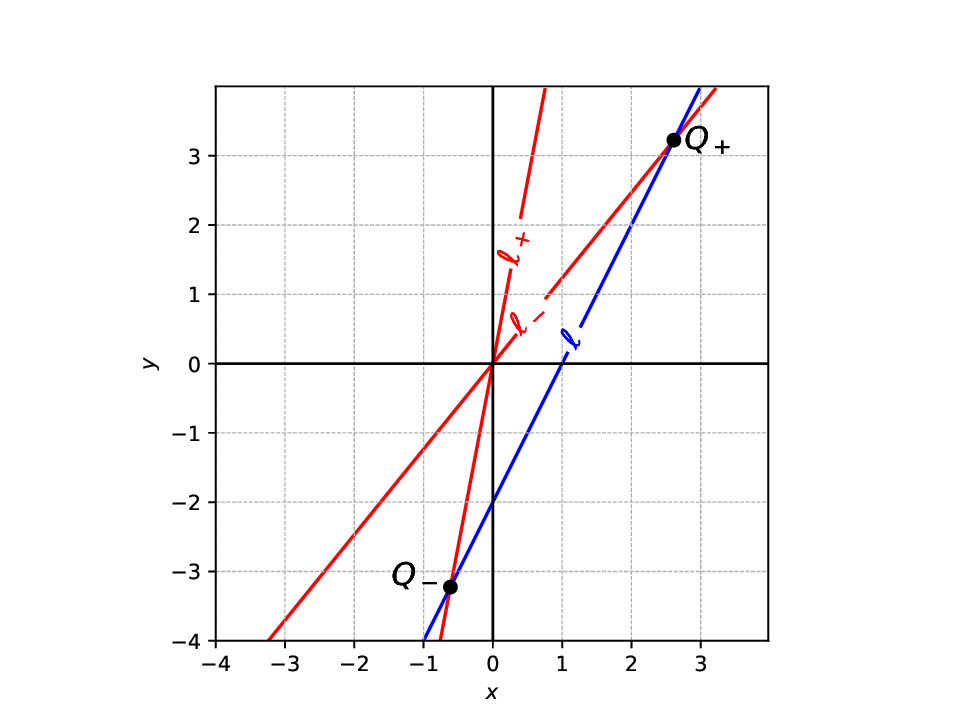}
			\caption{Lines $\ell$, $\ell_{+}$, and $\ell_{-}$ for $(g,r,d)=(14,3,13)$.}\label{fig:triangle_xy_bound}
		\end{figure}
		This defines a triangle bounded by the lines $\ell_{+}, \ell_{-}$ and $\ell$ for allowed $(x,y)$, with vertices $(0,0)$, \[Q_+ \coloneqq \ell\cap \ell_{-} = \left(1+\frac{d}{\sqrt{\abs{\Delta}}},\frac{2(g-1)}{\sqrt{\abs{\Delta}}}\right), \text{ and}\] \[Q_- \coloneqq \ell \cap \ell_{-} = \left( 1-\frac{d}{\sqrt{\abs{\Delta}}},- \frac{2(g-1)}{\sqrt{\abs{\Delta}}} \right).\]
		
		In particular, if $0<x$, then $x\leq 1+\frac{d}{\sqrt{\abs{\Delta}}}$ and $0<y\leq \frac{2(g-1)}{\sqrt{\abs{\Delta}}}$; and if $x\leq 0$, then $ 1-\frac{d}{\sqrt{\abs{\Delta}}}\leq x$ and $-\frac{2(g-1)}{\sqrt{\abs{\Delta}}}\leq y <0$, as claimed.
	\end{proof}

	\subsection{Distinguishing Brill--Noether loci via admissible assignments}\label{subsec: Dist BN via GF}
	
	To determine the relative positions of Brill--Noether loci, it is generally easier to prove non-containments $\BN{g}{r}{d}\nsubseteq \BN{g}{s}{e}$, as it suffices to find one curve admitting a $\g{r}{d}$ and not admitting a $\g{s}{e}$. To this end, we consider a K3 surface $S$ with $\Pic(S) = \Lambda^r_{g,d}$ and a smooth irreducible $C\in|H|$, so that $C\in \BN{g}{r}{d}$. If $C$ admits a $\g{s}{e}$, then $S$ carries an unstable Lazarsfeld--Mukai bundle $E_{C,\g{s}{e}}$, which has an admissible assignment. Clearly, if no such assignment exists, then $C$ admits no $\g{s}{e}$, and more generally the admissible assignments give bounds on $c_2(E_{c,\g{s}{e}})$ showing that $C$ admits no $\g{s}{e}$ for a range of $e$. Moreover, if there is an admissible assignment for some terminal filtration type, then one can in many cases use the admissible assignment to construct the Lazarsfeld--Mukai bundle $E_{C,\g{s}{e}}$. We note that the existence of admissible assignments it not sufficient, as there are examples of admissible assignments not arising from a terminal filtration, as in \Cref{ex:g_11_nonterm_adm_asaign,ex_3L_destab}.
	
	When one can construct such a Lazarsfeld--Mukai bundle, then it is natural to ask whether there is in fact a containment of Brill--Noether loci. Using notation from~\cite{arbarello_bruno_sernesi,ahk_2024}, let $\pi:\mathcal{P}_g\to\mathcal{K}_g$ be the universal smooth hyperplane section, the fiber above $(S,H)$ being the smooth irreducible curves in $|H|$, and let $\phi:\mathcal{P}_g \to \M_g$ be the forgetful map, and $\phi:\mathcal{P}^r_{g,d}\to\KK{g}{r}{d}$ the restriction to the divisor $\KK{g}{r}{d}$, which has a natural map $\phi:\mathcal{P}^r_{g,d}\to \BN{g}{r}{d}$, by~\cite[Lemma~1.7]{ahk_2024}. 
	
	\begin{defn}\label[defn]{defn:K3_expected_cont}
		We call the potential containment $\BN{g}{r}{d}\stackrel{?}{\subset}\BN{g}{s}{e}$ of Brill--Noether loci \emph{K3-expected} when $\phi(\mathcal{P}^r_{g,d})\subset \BN{g}{s}{e}$. That is, the potential containment is \emph{K3-expected} if for a general K3 surface $(S,H)\in\KK{g}{r}{d}$ and a smooth irreducible $C\in|H|$, we have $C\in\BN{g}{s}{e}$.
	\end{defn}
	
	We will not dwell on all containments which are K3-expected, though many of the containments below were first seen to be K3-expected, after which the containment was verified. Hence K3-expected containments are a useful heuristic in determining the relative positions of Brill--Noether loci.

	\section{Brill--Noether loci in genus 3--10}\label{sec: BN_strat_3_10}
	
	We identify the relative positions of Brill--Noether loci in genus $3$--$10$, noting some containments between components of certain Brill--Noether loci. We omit explicitly mentioning containments which occur as a result of Serre duality ($\BN{g}{r}{d}=\BN{g}{g-d+r-1}{2g-2-d}$) and the trivial containments of Brill--Noether loci ($\BN{g}{r}{d}\subseteq \BN{g}{r}{d+1}$ and $\BN{g}{r}{d}\subseteq \BN{g}{r-1}{d-1}$), and so restrict our attention to Brill--Noether loci with $d\leq g-1$. Moreover, as Clifford's Theorem shows that $\BN{g}{r}{2r}=\BN{g}{1}{2}$, for $d-2r=0$, we omit further mention of the loci $\BN{g}{r}{2r}$ for $r\geq 2$.
	
	As the Hasse diagrams of the Brill--Noether loci partially ordered by inclusion would be very cluttered, so we instead show a diagram of the ``minimal" non-trivial containments, i.e. the \emph{covers}, in the sense of partially ordered sets, which are not trivial containments. Recall that in a partially ordered set $(P,\leq)$, we say that $y$ \emph{covers} $x$ if $x\lneq y$ and there is no $z$ such that $x \lneq z \lneq y$.
	
	Thus any containment among Brill--Noether loci is a sequence of containments in the diagrams, trivial containments, and/or Serre duality. We arrange the Brill--Noether loci by $r$ on the horizontal axis and Clifford index $(d-2r)$ on the vertical axis. Trivial containments of the form $\BN{g}{r}{d}\subseteq\BN{g}{r}{d+1}$ give vertical arrows ($\uparrow$), and trivial containments of the form $\BN{g}{r}{d}\subseteq \BN{g}{r-1}{d-1}$ give northwest arrows ($\nwarrow$), both of which are generally omitted to avoid clutter, though some are included to give a better picture of the relative positions. When two Brill--Noether loci are equal, we denote this by $(=)$ in the diagrams. Finally, when a containment or non-containment is not known, e.g. $\BN{g}{r}{d}\stackrel{?}{\subset} \BN{g}{s}{e}$, we represent this with a dashed arrow $\BN{g}{r}{d}\dasharrow \BN{g}{s}{e}$ in the diagram.
	
	\begin{remark}
	We note that we have a non-containment $\BN{g}{1}{k}\nsubseteq \BN{g}{1}{\ell}$ for all $\ell <k$ by considering the general curve of a given gonality $k$, and we omit further mention of these non-containments.
	\end{remark}

	\subsection{Genus 3--6}\label{subsec: BNStrat genus 3_6} We show that in genus $3\leq g \leq 6$, the relative positions of Brill--Noether loci are determined completely by trivial containments.
	
	In genus $3$--$5$, note that the Brill--Noether loci with $d\leq g-1$ are only $\BN{g}{1}{d}$, and so the containments are all trivial. 
	
	\begin{theorem}
		The relative positions of Brill--Noether loci in genus $3\leq g \leq 6$ are given by trivial containments.
	\end{theorem}
	\begin{proof}
		As observed above, only genus $6$ requires argument. It is well-known that the Brill--Noether special curves of genus $6$ are either hyperelliptic, trigonal, or smooth plane quintics, the later having gonality $4$ by Max Noether's theorem, whereby $\BN{6}{2}{5}\nsubseteq \BN{6}{1}{3}$. Furthermore, one immediately sees that $\BN{6}{1}{2}\subset \BN{6}{2}{5}$  Noting that $\BN{6}{1}{2}=\BN{6}{3}{6}$, this is the trivial containment $\BN{6}{3}{6}\subset \BN{6}{2}{5}$. Furthermore, as $\kappa(6,2,5)=2$, \Cref{prop:maxk_distinguish} gives the non-containment $\BN{6}{1}{3}\nsubseteq \BN{6}{2}{5}$. Thus all containments in genus $6$ are given by trivial containments.
	\end{proof}
	
	\subsection{Genus 7}\label{subsec: BN Strat genus 7}
	
	The non-trivial containments of Brill--Noether loci with $\rho<0$ and $d\leq g-1$ are given in \Cref{fig:g_7_cont}.
	
	\begin{figure}[H]
		\[
		\xymatrix@R-0pc{
			\BN{7}{1}{4} & & \BN{7}{2}{6} \ar[ll] \\
			\BN{7}{1}{3} \ar[urr] & & \BN{7}{2}{5} \ar[ll] \\
			\BN{7}{1}{2} \ar@{=}[urr]&
		}
		\]
		
		\caption{Non-trivial containments of Brill--Noether loci in genus $7$.}\label{fig:g_7_cont}
	\end{figure}
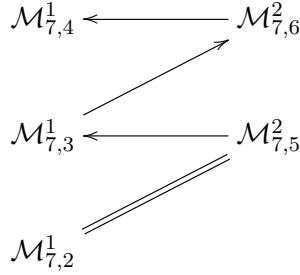
	
	We begin by noting containments, and then show all the required non-containments.
	
	\begin{prop}\label[prop]{prop:castelnuovo_hyperelleiptic_g_7}
		We have the following containments
		\begin{enumerate}[label={\normalfont(\roman*)}]
			\item $\BN{7}{2}{5}= \BN{7}{1}{2}$,
			\item $\BN{7}{2}{6}\subset \BN{7}{1}{4}$, and
			\item $\BN{7}{1}{3}\subset \BN{7}{2}{6}$.
		\end{enumerate}
	\end{prop}
	\begin{proof}
		Statement (i) follows from \Cref{lem:Cliff_1_cont}, (ii) follows from \Cref{lem:proj_from_nodes_plane}, and (iii) follows directly from \Cref{thm:rho_k}.
	\end{proof}
	
	It remains to prove the non-containments.
	
	\begin{prop}
		We have the following non-containments
		\begin{enumerate}[label={\normalfont(\roman*)}]
			\item $\BN{7}{1}{4} \nsubseteq \BN{7}{2}{6}$,
			\item $\BN{7}{1}{3} \nsubseteq \BN{7}{2}{5}$,
			\item $\BN{1}{2}{6}\nsubseteq \BN{7}{1}{3}$, and
			\item $\BN{7}{2}{6} \nsubseteq \BN{7}{2}{5}$.
		\end{enumerate}
	\end{prop}
	\begin{proof}
		We note that (i) and (ii) are a direct consequence of \Cref{prop:maxk_distinguish}. Statement (iv) is a consequence of (i) and (ii). Finally, (iii) follows by considering K3 surfaces with $\Pic(S)=\Lambda^2_{7,6}$, and noting that for a smooth irreducible curve $C\in|H|$, the Lazarsfeld--Mukai bundle $E_{C,\g{1}{e}}$ has only two admissible assignments with destabilizing subsheaf $H-L$ or $L$, both giving $c_2(E_{C,\g{1}{e}})=4$. Hence $C$ admits a $\g{2}{6}=|\calO_{C}(L)|$ and admits no $\g{1}{3}$.
	\end{proof}
	
	In total, we have identified the relative positions of Brill--Noether loci in genus $7$.
	
	\begin{theorem}\label[theorem]{thm:BNloci_genus_7_summary}
		All non-trivial containments among Brill--Noether loci in genus $7$ are obtained via trivial containments and the arrows appearing in \Cref{fig:g_7_cont}.
	\end{theorem}

	\subsection{Genus 8}\label{subsec: BN Strat genus 8}
	The non-trivial containments of Brill--Noether loci with $\rho<0$ and $d\leq g-1$ are given in \Cref{fig:g_8_cont}. We note that as in the genus $7$ case, some of the Brill--Noether loci of low Clifford index are all equal.
	
	\begin{figure}[H]
		\[
		\xymatrix@R-0pc{
			& & \BN{8}{2}{7}\\
			\BN{8}{1}{4} \ar[urr] & & \BN{8}{2}{6} \ar[ll]\\
			\BN{8}{1}{3} \ar[urr] & & \BN{8}{2}{5} \ar[ll] & & \BN{8}{3}{7} \ar@{=}[ll] \\
			\BN{8}{1}{2} \ar@{=}[urr] \ar@{=}[urrrr]
		}\]
		
		\caption{Non-trivial containments of Brill--Noether loci in genus $8$.}\label{fig:g_8_cont}
	\end{figure}
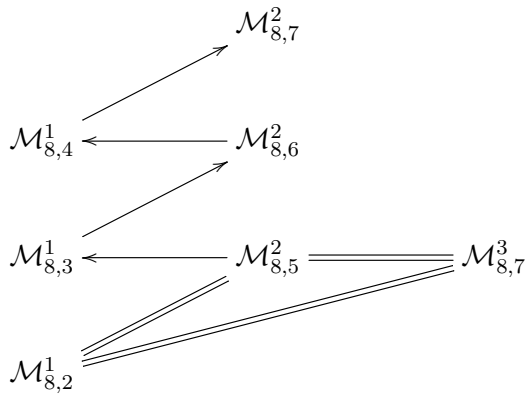

	We begin by noting the containments resulting from \Cref{lem:Cliff_1_cont} and the refined Brill--Noether theory for curves of fixed gonality, see \Cref{thm:rho_k}.
	
	\begin{prop}
		We have the containments
		\begin{enumerate}[label={\normalfont(\roman*)}]
			\item $\BN{8}{1}{2}=\BN{8}{2}{5}=\BN{8}{3}{7}$,
			\item $\BN{8}{1}{3}\subseteq \BN{8}{2}{6}$, and
			\item$\BN{8}{1}{4}\subseteq \BN{8}{2}{7}$.
		\end{enumerate}
	\end{prop}

It remains to show the non-containments. We begin with the noncontainments which are a direct result of \Cref{prop:maxk_distinguish}.

\begin{prop}
	We have the non-containments 
	\begin{enumerate}[label={\normalfont(\roman*)}]
		\item $\BN{8}{1}{4}\nsubseteq \BN{8}{2}{6}$, and
		\item $\BN{8}{2}{7}\nsubseteq \BN{8}{2}{6}$.
	\end{enumerate}
\end{prop}

The last remaining containments are shown using K3 surfaces.

\begin{prop}
	We have the noncontainments
	\begin{enumerate}[label={\normalfont(\roman*)}]
		\item $\BN{8}{2}{6}\nsubseteq \BN{8}{1}{3}$,
		\item $\BN{8}{2}{7}\nsubseteq \BN{8}{1}{4}$.
	\end{enumerate}
\end{prop}
\begin{proof}
	We consider curves on K3 surfaces with $\Pic(S)=\Lambda^2_{8,d}$ for $d=6,7$, and consider admissible assignments for the Lazarsfeld--Mukai bundle $E_{C,\g{1}{e}}$. One can easily check that when $d=6$, there is only one admissible assignment with destabilizing subsheaf $H-L$, giving $c_2(E_{C,\g{1}{e}})=4$, and (i) follows. For $d=7$, there are exactly two admissible assignments, with destabilizing subsheaf $H-L$ or $L$, both giving $c_2(E_{C,\g{1}{e}})=5$, from which (ii) follows. 
\end{proof}

In total, we have identified the relative positions of Brill--Noether loci in genus $8$.

\begin{theorem}\label[theorem]{thm:BNloci_genus_8_summary}
	All non-trivial containments among Brill--Noether loci in genus $8$ are obtained via trivial containments and the arrows appearing in \Cref{fig:g_8_cont}.
\end{theorem}

	\subsection{Genus 9}\label{subsec: BN Strat genus 9}
	The non-trivial containments of Brill--Noether loci with $\rho<0$ and $d\leq g-1$ are given in \Cref{fig:g_9_cont}.
	
	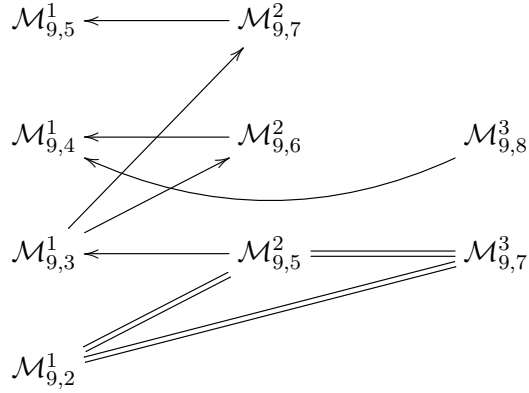
\begin{figure}[H]
		\[
		\xymatrix{
			\BN{9}{1}{5} & & \BN{9}{2}{7} \ar[ll] \\
			\BN{9}{1}{4} & & \BN{9}{2}{6} \ar[ll] & & \BN{9}{3}{8} \ar@/^2pc/[llll] \\
			\BN{9}{1}{3} \ar[urr] \ar[uurr] & & \BN{9}{2}{5} \ar@{=}[dll] \ar[ll] & & \BN{9}{3}{7} \ar@{=}[dllll] \ar@{=}[ll] \\
			\BN{9}{1}{2} 
		}\]
		\caption{Non-trivial containments of Brill--Noether loci in genus $9$.}\label{fig:g_9_cont}
	\end{figure}
	
	We begin by noting the containments resulting from \Cref{lem:Cliff_1_cont} and the refined Brill--Noether theory for curves of fixed gonality, see \Cref{thm:rho_k}.
	
	\begin{prop}
		We have the following containments
		\begin{enumerate}[label={\normalfont(\roman*)}]
			\item $\BN{9}{1}{2} = \BN{9}{2}{5} = \BN{9}{3}{7}$,
			\item $\BN{9}{1}{2}\subset \BN{9}{3}{8}$,
			\item $\BN{9}{1}{3} \subset \BN{9}{2}{6}$, and
			\item $\BN{9}{1}{3} \subset \BN{9}{2}{7}$.
		\end{enumerate}
	\end{prop}

There are additional containments coming from projections from nodes of plane models.

\begin{prop}
	We have the containments 
	\begin{enumerate}[label={\normalfont(\roman*)}]
		\item $\BN{9}{2}{6}\subset \BN{9}{1}{4}$, and
		\item $\BN{9}{2}{7} \subset \BN{9}{1}{5}$.
	\end{enumerate}
\end{prop}
\begin{proof}
		Both (i) and (ii) follow from \Cref{lem:proj_from_nodes_plane}, we note that (ii) is~\cite[Proposition~6.4]{auel_haburcak_2022}.
\end{proof}

\begin{prop}
	We have the containment $\BN{9}{3}{8}\subset \BN{9}{1}{4}$.
\end{prop}
\begin{proof}
	The proof is delayed to \Cref{prop:g^3_9_genus9_containments}.
\end{proof}

It remains to show the non-containments. We begin with a containment and a non-containment that are shown by considering complete intersections of a K3 surface of Picard rank $1$, due to Mori~\cite{Mori_complete_int}. In particular, the proof of~\cite[Theorem~1]{Mori_complete_int} shows that if $g$ and $d$ satisfy $8(g-1)=d^2$, then there is a smooth curve of genus $g$ and degree $d$ in $\PP^3$ which is the intersection of a K3 surface $(S,H)$ of genus $3$ with $\Pic(S)=\Z[H]$ and a hypersurface of degree $d/4$.

\begin{prop}\label[prop]{prop:g^3_9_genus9_containments}
	There is a containment $\BN{9}{3}{8}\subset \BN{9}{1}{4}$ and a non-containment $\BN{9}{3}{8}\nsubseteq \BN{9}{2}{6}$.
\end{prop}
\begin{proof}
	Let $C\in\BN{9}{3}{8}$. We first suppose that the $\g{3}{8}$ is not very ample, and instead factors through a map $C\to D \subset \PP^3$, where $D$ is a curve of arithmetic genus $\gamma$. By \Cref{thm:lange_birat_dim}, we see that the only possibility is that $\gamma=0$, and the map $C\to D$ is either degree $2$ or degree $4$, the latter cannot occur, as then $D\to \PP^3$ would be given by a $\g{3}{2}$. Hence if the $\g{3}{8}$ is not very ample, then the curve $C$ is hyperelliptic, hence $C\in\BN{9}{1}{2}\subset \BN{9}{1}{4}$.
	
	We suppose now that the $\g{3}{8}$ is very ample, and $C\subset \PP^3$ is a smooth curve of genus $9$ and degree $8$. One can easily show that in this case $C$ lies on a smooth quadric in $\PP^3$ as a curve of type $(4,4)$. Moreover, as $8(g-1)=d^2$, Mori shows in~\cite{Mori_complete_int} that $C$ is the complete intersection of a quartic and a quadric in $\PP^3$, the quartic being a K3 surface $(S,H)$ of genus $3$ with $\Pic(S)=\Z[H]$, and $C\in|2H|$ a smooth curve, with $\calO_C(H)$ the $\g{3}{8}$. As $C$ is a complete intersection, $\gon(C)\geq 4$ by \Cref{thm:gon_comp_int}, hence is not hyperelliptic. In fact, $C$ has a $4$-secant line, showing that $\gon(C)=4$ and $C\in\BN{9}{1}{4}$. Thus for $C\in\BN{9}{3}{8}$, $C$ has a $\g{1}{4}$ whether or not the $\g{3}{8}$ is very ample, whereby $\BN{9}{3}{8}\subset \BN{9}{1}{4}$.
	
	Finally, we show that a smooth curve $C\in |2H|$ cannot have a $\g{2}{6}$. Suppose for contradiction that $C$ admits a $\g{2}{6}$, and let $E=E_{C,\g{2}{6}}$ be its Lazarsfeld--Mukai bundle, we have $\rk(E)=3$, $c_1(E)=2H$, and $c_2(E)=6$. We analyze the possible terminal filtrations of $E$, and show there are none.
	
	Suppose first that $E$ has a terminal filtration of type $2\subset 3$ or $1\subset 2 \subset 3$, then $E$ sits in an exact sequence of the form \[0 \to M \to E \to N \to 0,\] where $\rk(M)=2$, $\rk(N)=1$, $c_1(N)=aH$, $c_1(M)=bH$, such that $a+b=2$ and \[\mu(M)=2a \geq \mu(E) = \frac{8}{3} \geq \mu(N)=4b>0.\] Thus $\frac{2}{3}\geq b>0$, which is a contradiction.
	
	Suppose now that $E$ has a terminal filtration of type $1\subset 3$. Then $E$ sits in an exact sequence \[0 \to N \to E \to M \to 0,\] where $\rk(N)=1$, $\rk(M)=2$, and as before $c_1(N)=aH$, $c_1(M)=bH$, $a+b=2$, and now \[\mu(N)=4a \geq \mu(E)=\frac{8}{3} \geq\mu(M) = 2b>0.\] Thus $a=b=1$, and $c_2(M)=2$. However, as $M$ is stable, we must have \[c_2(M)\geq \frac{(\rk(M)-1)c_1(M)^2}{2\rk(M)}+\rk(M)-\frac{1}{\rk(M)},\] which gives $2\geq \frac{(2-1)(4)}{4}+2-\frac{1}{2}=3-\frac{1}{2}$, a contradiction.
	
	As there are no terminal filtrations for the Lazarsfeld--Mukai bundle of a $\g{2}{6}$, $C$ admits no $\g{2}{6}$, showing $\BN{9}{3}{8}\nsubseteq \BN{9}{2}{6}$, as claimed.
\end{proof}

We give the non-containments that are a direct result of \Cref{prop:maxk_distinguish}.

\begin{prop}
	We have the non-containments
	\begin{enumerate}[label={\normalfont(\roman*)}]
		\item $\BN{9}{1}{4}\nsubseteq \BN{9}{2}{6}$,
		\item $\BN{9}{1}{4}\nsubseteq \BN{9}{3}{8}$,
		\item $\BN{9}{1}{5}\nsubseteq \BN{9}{2}{6}$,
		\item $\BN{9}{1}{5}\nsubseteq \BN{9}{2}{7}$,
		\item $\BN{9}{2}{6}\nsubseteq \BN{9}{3}{8}$, and
		\item $\BN{9}{2}{7}\nsubseteq \BN{9}{3}{8}$.
	\end{enumerate}
\end{prop}
\begin{proof}
	The non-containments follow directly from \Cref{prop:maxk_distinguish}.
\end{proof}

The remaining non-containments  come from considering admissible assignments on K3 surfaces of Picard rank $2$.

\begin{prop}\label[prop]{prop:g9_non_cont_k3s}
	We have the non-containments
	\begin{enumerate}[label={\normalfont(\roman*)}]
		\item $\BN{9}{2}{6}\nsubseteq \BN{9}{1}{3}$,
		\item $\BN{9}{2}{7}\nsubseteq \BN{9}{2}{6}$, and
		\item $\BN{9}{2}{7}\nsubseteq \BN{9}{1}{4}$.
	\end{enumerate} 
\end{prop}
\begin{proof}
	To prove (i), we consider a K3 surface $(S,H)$ with $\Pic(S)=\Lambda^2_{9,6}$, and find the possible admissible assignments of a Lazarsfeld--Mukai bundle of a $\g{1}{e}$ on a smooth irreducible curve $C\in|H|$. The only possible admissible assignments are for a terminal filtration of type $1\subset 2$, with \[c_1(E_1)\in\{ 2H-4L, H-L, 2L, -H+4L\},\] giving bounds $e\geq 8,4,4,8$, respectively. Hence $C$ admits no $\g{1}{e}$ with $e\leq 3$. An alternative proof can be obtained using \Cref{thm:coppens_plane_gon}.
	
	To prove (ii), we proceed analogously, and consider a K3 surface $(S,H)$ with $\Pic(S)=\Lambda^2_{9,7}$, and find the possible admissible assignments of a Lazarsfeld--Mukai bundle of a $\g{2}{e}$ on a smooth curve $C\in |H|$. The only admissible assignments are for a terminal filtration type of $1\subset 3$, as can be easily checked by showing that destabilizing subsheaves of higher rank cannot satisfy \Cref{eq:term_fil_quot_slope_ineq} and \Cref{eq:term_filt_quot_nonneg_selfint}. Furthermore, the only admissible assignments satisfying \Cref{lem:ineq_destab_fil} have destabilizing sub-line bundle either $H-L$ or $L$, giving a bound on $c_2(E_{C,\g{2}{e}})$ of $\geq 7$ and $\geq 7.5$, respectively. Thus $C$ admits no $\g{2}{e}$ with $e\leq 6$.
	
	To prove (iii), we consider K3 surfaces as in the proof of case (ii), and see that the only admissible assignments of $E_{C,\g{1}{e}}$ are of type $1\subset 2$, with destabilizing sub-line bundle $H-L$, each giving $c_2(E_{C,\g{1}{e}})\geq 5$. Thus $C$ admits no $\g{1}{e}$ with $e\leq 4$. Again, \Cref{thm:coppens_plane_gon} provides an alternative proof.
\end{proof}
\begin{remark}
	An alternative proof of \Cref{prop:g9_non_cont_k3s}~(iii) is as a corollary of \Cref{prop:maxk_distinguish} and dimension arguments. Assume for contradiction that $\BN{9}{2}{7}\subseteq \BN{9}{1}{4}$. We note that since \[\rho(9,1,4)=\rho(9,2,7)=-3,\] $\BN{9}{1}{4}$ is irreducible of codimension $3$ and \cite[Theorem~0.1]{steffen_1998} shows that every component of $\BN{9}{2}{7}$ has codimension at most $3$. By assumption, $\BN{9}{2}{7}\subseteq \BN{9}{1}{4}$, hence every component of $\BN{9}{2}{7}$ has codimension $3$ and $\BN{9}{2}{7}$ must in fact be irreducible, whereby $\BN{9}{2}{6}=\BN{9}{1}{4}$. However, we compute that $\maxk(9,2,7)=3$, whereby $\BN{9}{1}{4} \nsubseteq\BN{9}{2}{7}$, a contradiction. Therefore $\BN{9}{2}{7}\nsubseteq \BN{9}{1}{4}$.
\end{remark}

In total, we have identified the relative positions of Brill--Noether loci in genus $9$.

\begin{theorem}\label[theorem]{thm:BNloci_genus_9_summary}
	All non-trivial containments among Brill--Noether loci in genus $9$ are obtained via trivial containments and the arrows appearing in \Cref{fig:g_9_cont}.
\end{theorem}

	\subsection{Genus 10}\label{subsec: BN Strat genus 10}
	The non-trivial containments of Brill--Noether loci with $\rho<0$ and $d\leq g-1$ are given in \Cref{fig:g_10_cont}.
	
	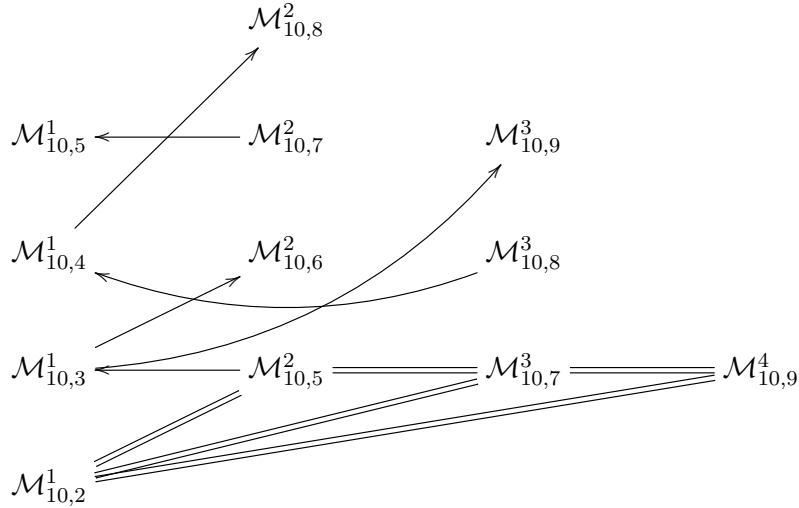
\begin{figure}[!h]
		\[
		\xymatrix{
			 & & \BN{10}{2}{8} \\
			\BN{10}{1}{5} & & \BN{10}{2}{7} \ar[ll]  & & \BN{10}{3}{9}\\
			\BN{10}{1}{4} \ar[uurr] & & \BN{10}{2}{6} & & \BN{10}{3}{8} \ar@/^1.7pc/[llll] \\
			\BN{10}{1}{3} \ar[urr] \ar@/_2pc/[uurrrr]  & & \BN{10}{2}{5} \ar@{=}[rr] \ar[ll]  & & \BN{10}{3}{7} & & \BN{10}{4}{9} \ar@{=}[ll] \\
			\BN{10}{1}{2} \ar@{=}[urr] \ar@{=}[urrrr] \ar@{=}[urrrrrr]
		}\]
		\caption{Non-trivial containments of Brill--Noether loci in genus $10$.}\label{fig:g_10_cont}
	\end{figure}
	
	We begin by noting the containments resulting from \Cref{lem:Cliff_1_cont} and the refined Brill--Noether theory for curves of fixed gonality, see \Cref{thm:rho_k}.
	
	\begin{prop}
		We have 
		\begin{enumerate}[label={\normalfont(\roman*)}]
			\item $\BN{10}{1}{2}=\BN{10}{2}{5}=\BN{10}{3}{7}=\BN{10}{4}{9}$
			\item $\BN{10}{1}{3}\subset \BN{10}{2}{6}$
			\item $\BN{10}{1}{3}\subset \BN{10}{3}{9}$
			\item $\BN{10}{1}{4}\subset \BN{10}{2}{7}$.
		\end{enumerate}
	\end{prop}

	\begin{prop}
		We have the containment $\BN{10}{2}{7}\subseteq \BN{10}{1}{5}$.
	\end{prop}
	\begin{proof}
		This follows directly from \Cref{lem:proj_from_nodes_plane}.
	\end{proof}
	\begin{remark}
		We note that~\cite[Proposition~4.2]{h_tib_theta} shows that $\BN{10}{2}{7}$ is irreducible.
	\end{remark}
	\begin{prop}\label[prop]{prop:bn10_3_8_subset_bn10_1_4}
		We have the containment $\BN{10}{3}{8} \subseteq \BN{10}{1}{4}$.
	\end{prop}
	\begin{proof}
		Let $C\in \BN{10}{3}{8}$. We see from \Cref{thm:castelnuovo_bound} that the $\g{3}{8}$ cannot be both basepoint free and birationally very ample. If it is not basepoint free, then $C\in\BN{10}{3}{7}=\BN{10}{1}{2}\subset \BN{10}{1}{4}$, as claimed. We may suppose that the $\g{3}{8}$ is basepoint free and instead factors as a map \[C \stackrel{k:1}{\longrightarrow}D \longrightarrow \PP^3,\] where $D$ is a smooth curve of genus $\gamma$ and $k=2,4$. From \Cref{thm:lange_birat_dim}, we see that if $k=2$, then $\gamma\leq 2$; and if $k=4$, then $\gamma\leq 1$. However, we cannot have $k=2$ and $\gamma=2$, as then $D$ admits a $\g{3}{4}$, which a curve of genus $2$ does not. Likewise, we cannot have $k=4$ and $\gamma=1$, as then $D$ would admit a $\g{3}{2}$. Hence either $k=2$ and $\gamma\leq 1$, or $k=4$ and $\gamma=1$. Hence $C$ is either hyperelliptic, bielliptic (and admits a $\g{1}{4}$ pulling back the hyperelliptic map $D\to \PP^1$), or admits a $\g{1}{4}$. In any case, $C\in\BN{10}{1}{4}$, as claimed.
	\end{proof}
	
	It remains to show the non-containments. We start with the non-containments that are a result of \Cref{prop:maxk_distinguish}.
	
	\begin{prop}
		We have 
		\begin{enumerate}[label={\normalfont(\roman*)}]
			\item $\BN{10}{1}{3}\nsubseteq \BN{10}{2}{5}$, $\BN{10}{1}{3} \nsubseteq \BN{10}{3}{8}$, $\BN{10}{1}{4} \nsubseteq \BN{10}{2}{6}$, $\BN{10}{1}{4} \nsubseteq \BN{10}{2}{7}$, $\BN{10}{1}{4} \nsubseteq \BN{10}{3}{8}$,\newline $\BN{10}{1}{4} \nsubseteq \BN{10}{3}{9}$, $\BN{10}{1}{5} \nsubseteq \BN{10}{2}{7}$, $\BN{10}{1}{5} \nsubseteq \BN{10}{3}{9}$, $\BN{10}{1}{5} \nsubseteq \BN{10}{2}{8}$

			\item $\BN{10}{2}{6} \nsubseteq \BN{10}{3}{8}$, $\BN{10}{2}{7} \nsubseteq \BN{10}{3}{8}$, $\BN{10}{2}{7} \nsubseteq \BN{10}{3}{9}$, $\BN{10}{2}{8} \nsubseteq \BN{10}{2}{7}$, $\BN{10}{2}{8} \nsubseteq \BN{10}{3}{9}$
		\end{enumerate}
	\end{prop}
	\begin{proof}
		The non-containments in follow directly from \Cref{prop:maxk_distinguish}.
	\end{proof}
	
	\begin{prop}
		We have the non-containment $\BN{10}{2}{6}\nsubseteq \BN{10}{1}{4}$.
	\end{prop}
	\begin{proof}
		A smooth curve of degree $6$ in $\PP^2$ has genus $10$. By Max Noether's theorem on smooth plane curves, such a curve has gonality $5$, the $\g{1}{5}$ obtained by projecting from a general point.
	\end{proof}
	
	We give two proofs that a smooth plane curve of degree $6$ admits no $\g{3}{9}$.
	
	\begin{prop}\label[prop]{prop:10_2_6_no_10_3_9}
		We have $\BN{10}{2}{6}\nsubseteq \BN{10}{3}{9}$.
	\end{prop}
	\begin{proof}[{Proof~1 of \Cref{prop:10_2_6_no_10_3_9}}]
		Let $C$ be a smooth degree $6$ curve in $\PP^2$. By the above results, we may assume that the $\g{3}{9}$ is complete, given by a line bundle $A\in\Pic^9(C)$. Projecting from a point not on the curve gives a degree $6$ map $f: C\to \PP^1$, and by~\cite[Theorem~1.1]{larson_2024brillnoethertheorysmoothcurves}, if $C$ admits a $\g{3}{9}$, then \[f_{\ast}(A)\cong\calO_{\PP^1}(e_1)\oplus \cdots \oplus \calO_{\PP^1}(e_6), \text{ with } \sum_{i=1}^6 e_i = -6,\ e_{i+1}-e_{i}\leq 1, \text{ and } h^0(\PP^1,\calO_{\PP^1}(\vec{e}))=4.\] Enumerating all possibilities for $e_i$ shows that these conditions cannot be satisfied. Namely, as $4=h^0(C,A)=h^0(\PP^1,\calO_{\PP^1}(\vec{e}))$, the possibilities are
		\begin{enumerate}[label=(\roman*)]
			\item $(e_1,e_2,e_3,e_4,e_5,3)$ with $e_1\leq\cdots e_5\leq -1$ (but $e_6-e_5\geq 2$),
			\item $(e_1,e_2,e_3,e_4,0,2)$ with $e_1\leq \cdots e_4\leq -1$ (but $e_6-e_5\geq 2$),
			\item $(e_1,e_2,e_3,e_4,1,1)$ with $e_1\leq \cdots e_4\leq -1$ (but $e_5-e_4\geq 2$),
			\item $(e_1,e_2,e_3,0,0,1)$ with $e_1\leq e_2\leq e_3\leq -1$,
			\item $(e_1,e_2,0,0,0,0)$ with $e_1\leq e_2\leq -1$.
		\end{enumerate}
		As noted, cases (i)--(iii) do not occur. One can readily check that in case (iv) and (v), as $e_1+e_2+e_3=-7$ and $e_1+e_2=-6$, respectively, the condition $e_{i+1}-e_i\leq 1$ cannot be satisfied. Thus, in fact, the general smooth plane curve of degree $6$ admits no $\g{3}{9}$.
	\end{proof}
	
	The second proof comes from~\cite{Coppens_Kato_nontriv_plane_94}, where Coppens--Kato define a base point free complete linear series $\g{r}{n}$ on a smooth plane curve of degree $d$ to be \emph{trivial} if $r\geq 1$, $\dim|\omega_C-\g{r}{n}|\geq 1$, and $\g{r}{n}=|m\g{2}{d}-D|$ for some $m\geq 0$ and some effective divisor $D$ (of degree $md-n$) on $C$ such that \[r=\frac{m^2+3m}{2}-md+n;\] and show that a non-trivial complete basepoint free $\g{r}{n}$ satisfies \[n\geq (d-3)(x+1)-\beta, \text{ where } r=\frac{(x+1)(x+2)}{2}-\beta \text{ for } x\geq1, 0\leq \beta <x.\]
	
	\begin{proof}[{Proof~2 of \Cref{prop:10_2_6_no_10_3_9}}]
		As $C$ is not hyperelliptic, we may assume that the $\g{3}{9}$ is complete as $C\notin\BN{10}{4}{9}=\BN{10}{1}{2}$, and that the fixed part of the $\g{3}{9}$ has degree $1$ as $C\notin \BN{10}{3}{7}=\BN{10}{1}{2}$. 
		
		One can readily check that a basepoint free $\g{3}{9}$ is non-trivial, as there are no integer solutions for $m$, but then $9\geq (6-3)(3+1)=12$, hence $C$ admits no complete base point free $\g{3}{9}$. If the $\g{3}{9}$ is not basepoint free, one can readily check that a $\g{3}{8}$ is also non-trivial and likewise does not satisfy the bounds.
	\end{proof}
	
	\begin{remark}
		We note that $\BN{10}{2}{6}$ is reducible (as it contains $\BN{10}{1}{3}$ which is of larger than expected dimension and a component coming from the image of the Severi variety whose general element has gonality $5$). The locus $\BN{10}{1}{3}$, which is likely a component of $\BN{10}{2}{6}$, is contained in $\BN{10}{3}{9}$.
	\end{remark}
	
	\begin{prop}\label[prop]{prop:genus_10_bielleiptic_3_9_no_2_6}
		We have $\BN{10}{3}{8}\nsubseteq \BN{10}{2}{6}$.
	\end{prop}
	\begin{proof}
		We first note that there are curves of genus $10$ with a $\g{3}{8}$ and no $\g{1}{2}$ nor $\g{1}{3}$, for example a double cover of an elliptic curve ramified at 18 points. Such curves exist by \Cref{thm:lange_birat_dim}, admit a $\g{3}{8}$ by puling back the $\g{3}{4}$ on the elliptic curve, and \Cref{thm:castelnuovo_severi} shows they admit no $\g{1}{2}$ or $\g{1}{3}$.
		
		Suppose for contradiction that $\BN{10}{3}{8}\subseteq \BN{10}{2}{6}$, and let $C\in \BN{10}{3}{8}$ be a non-hyperelliptic curve. We first claim that the $\g{2}{6}$ cannot be very ample. Indeed, if the $\g{2}{6}$ were very ample, then $C$ would be a smooth plane curve of degree $6$ and would have gonality $5$, contradicting $\BN{10}{3}{8}\subseteq \BN{10}{1}{4}$, cf. \Cref{prop:bn10_3_8_subset_bn10_1_4}. As $C$ is not hyperelliptic, the $\g{2}{6}$ is basepoint free, and must factor through a lower genus curve. Again, as $C$ is not hyperelliptic, it follows from \Cref{thm:lange_birat_dim} that the $\g{2}{6}$ must factor as a $\g{1}{3}$ followed by a degree $2$ map $\PP^1\to \PP^2$. Hence every $C\in \BN{10}{3}{8}$ in fact has a $\g{1}{3}$. However, as noted above, a genus $10$ double cover of an elliptic curve is in $\BN{10}{3}{8}$, but admits no $\g{1}{3}$, a contradiction. Therefore $\BN{10}{3}{8}\nsubseteq \BN{10}{2}{6}$, as claimed.
	\end{proof}
	
	The remaining non-containments are proven using K3 surfaces or chains of elliptic curves.
	
	\begin{prop}
		We have the non-containment $\BN{10}{2}{7}\nsubseteq \BN{10}{3}{9}$.
	\end{prop}
	\begin{proof}
		From~\cite[Proposition~4.2]{h_tib_theta}, we see that $\BN{10}{2}{7}$ is irreducible, and its unique component is of expected dimension, and obtained by smoothing chains of elliptic curves as in~\cite[Theorem~2.1]{bigas2023brillnoether}. Thus the general curve in $\BN{10}{2}{7}$ admits a degeneration to a chain of $10$ elliptic curves and a limit $\g{2}{7}$ given by the admissible filling in \Cref{fig:g27_filling}.
		
		\begin{figure}[H]
			\begin{tikzpicture}[scale=.45]
				
				\begin{scope}[ ]
					\foreach \x in {0, 1,2,3,4,5} {\draw[thick] (0,\x) -- (3, \x); }
					\foreach \x in {0, 1,2,3} {\draw[thick] (\x,0) -- ( \x,5); }
					
					\node at (0.5, 4.5) {1};
					\node at (1.5, 4.5) {2};
					\node at (2.5, 4.5) {3};
					
					\node at (0.5, 3.5) {2};
					\node at (1.5, 3.5) {4};
					\node at (2.5, 3.5) {6};
					
					\node at (0.5, 2.5) {3};
					\node at (1.5, 2.5) {5};
					\node at (2.5, 2.5) {7};
					
					\node at (0.5, 1.5) {6};
					\node at (1.5, 1.5) {8};
					\node at (2.5, 1.5) {9};
					
					\node at (0.5, 0.5) {7};
					\node at (1.5, 0.5) {9};
					\node at (2.5, 0.5) {10};	
				\end{scope}
				
			\end{tikzpicture}
			\caption{  Admissible filling giving a $\g{2}{7}$. }
			\label{fig:g27_filling}
		\end{figure}
		
		 In particular, the curves $E_2$, and $E_9$ have a $2$-torsion condition on the nodes, and $E_3$, $E_5$, and $E_7$ have a $4$-torsion condition on the nodes. There is no admissible filling with the entries $1,\dots,10$ of a $4\times 4$ square with the same torsion conditions, hence the chain of elliptic curves admits no limit $\g{3}{9}$, and hence the general curve in $\BN{g}{2}{7}$ admits no $\g{3}{9}$, as claimed.
	\end{proof}
	
	\begin{prop}
		We have the non-containment $\BN{10}{3}{9}\nsubseteq \BN{10}{1}{5}$.
	\end{prop}
	\begin{proof}
		As in~\cite[Exercise~1.41]{harris_mumford} and~\cite[Remark~5.2]{h_tib_theta}, there is a component of $\BN{10}{3}{9}$ where the $\g{3}{9}$ is a theta characteristic. The general such curve is a complete intersection of two cubics in $\PP^3$, hence by \Cref{thm:gon_comp_int} has gonality $6$.
	\end{proof}
	
	We now give the non-containments using admissible assignments on K3 surfaces.
	
	\begin{prop}
		We have the non-containments 
		\begin{enumerate}[label={\normalfont(\roman*)}]
			\item $\BN{10}{2}{7} \nsubseteq \BN{10}{1}{4}$
			\item $\BN{10}{2}{7} \nsubseteq \BN{10}{2}{6}$
			\item $\BN{10}{3}{9} \nsubseteq \BN{10}{2}{7}$
			\item $\BN{10}{3}{9} \nsubseteq \BN{10}{3}{8}$
		\end{enumerate}
	\end{prop}
	\begin{proof}
		To show (i) and (ii), we consider a K3 surface with $\Pic(S)=\Lambda^2_{10,7}$, and note that by \Cref{lem:ineq_destab_fil}, if a smooth curve $C\in|H|$ admits a $\g{1}{k}$ or a $\g{2}{e}$, then any destabilizing subsheaf $N$ of $E_{C,\g{1}{k}}$ or $E_{C,\g{2}{e}}$ must have $c_1(N)=(1-x)H+yL$ with $0\leq \abs{x} \leq 2$ and $0\leq \abs{y}\leq 4$. Checking all possibilities for destabilizing sub-line bundles of $E_{C,\g{1}{k}}$, one sees that the only admissible assignment has $c_1(N)=(H-L)$, giving $c_2(E_{C,\g{1}{k}})\geq 5$, showing (i). Likewise, the only admissible assignments for $E_{C,\g{2}{e}}$ are of type $1\subset 3$, and have $c_1(N)=H-L$, or $c_1(N)=L$,  giving $c_1(E_{C,\g{2}{e}})\geq 7$, or $\geq 8$, respectively, whereby (ii) follows.
		
		To show (iii) and (iv), we consider a K3 surface with $\Pic(S)=\Lambda^3_{10,9}$. We see that the only terminal filtrations of a Lazarsfeld--Mukai bundle of type $E_{C,\g{2}{e}}$ are of type $1\subset 3$, with destabilizing sub-line bundle $N\in \{2H-3L, H-L, L, -H+3L \}$, giving $c_2(E_{C,\g{2}{e}})\geq 10.5, 7.5, 7.5, 10.5$, respectively, and (iii) follows. Similarly, considering destabilizing filtrations of $E_{C,\g{3}{e}}$, we see from \Cref{lem:ineq_destab_fil} and \Cref{eq:term_fil_quot_slope_ineq} that the destabilizing filtration must have type $1\subset 4$, $2\subset 4$, or $1\subset 2\subset 4$. One can check that there are no admissible assignments of type $1\subset 2 \subset 4$. And checking all admissible assignments of type $1\subset 4$ and $2\subset 4$, one finds that the destabilizing subsheaf has $c_1(N)\in \{2H-3L, H-L, L, -H+3L \}$, for both $\rk(N)=1$ and $\rk(N)=2$, giving $c_2(E_{C,\g{3}{e}})\geq 11+\frac{1}{3},9,9,11+\frac{1}{3}$, and $\geq 12,10,10,12$, respectively, and (iv) follows.
	\end{proof}
	
	In total, we have identified the relative positions of Brill--Noether loci in genus $10$.
	
	\begin{theorem}\label[theorem]{thm:BNloci_genus_10_summary}
		All non-trivial containments among Brill--Noether loci in genus $10$ are obtained via trivial containments and the arrows appearing in \Cref{fig:g_10_cont}.
	\end{theorem}

	\section{Brill--Noether loci in genus \texorpdfstring{$\geq 11$}{}}\label{sec:BN_strat_ge_11}
	
	In genus $g\geq 11$, the previous methods for distinguishing Brill--Noether loci are no longer sufficient give a complete picture of all containments of Brill--Noether loci. For example, our strategy using K3 surfaces with $\Pic(S)=\Lambda^r_{g,d}$ fails when $\Delta(g,r,d)>0$, as the Hodge index theorem shows there are no K3 surfaces with such a Picard group where $H$ is ample. Nevertheless, we can elucidate the relative positions of many Brill--Noether loci, and in genus $11$ we give a complete picture, see \Cref{subsec: BN Strat genus 11}.
	
	\subsection{Clifford index \texorpdfstring{$\leq 3$}{} and bielliptic curves}\label{subsec:cliff_2_bielliptic_g26_containments}
	
	For distinguishing Brill--Noether loci of small Clifford index, small degree covers of curves are quite useful. We collect a few results.
	
	\begin{lemma}\label[lemma]{lem:Cliff_2_bielliptic}
		Suppose either that \begin{itemize}
			\item $r=2$ and $g\geq 11$, or 
			\item $r\geq 3$ and $g\geq 7$.
		\end{itemize} If $d-2r\geq2$, then $\BN{g}{r}{d}\nsubseteq \BN{g}{1}{3}$.
	\end{lemma}
	\begin{proof}
		By the trivial containments, it suffices to prove the result for $d-2r=2$.
		
		Consider bielliptic curves $C\stackrel{2:1}{\longrightarrow} E$, which exist by \Cref{thm:lange_birat_dim}. We note that such curves $C$ admit a $\g{e-1}{2e}$ for all $e\geq 2$, given by the composition \[C \stackrel{2:1}{\longrightarrow} E \stackrel{\g{e-1}{e}}{\longrightarrow}\PP^{e-1}.\] From \Cref{thm:castelnuovo_severi}, we see that $C$ admits no $\g{1}{3}$, as the $\g{1}{3}$ cannot factor through the map $C\to E$ and $g(C)\geq 5$, whereby $\BN{g}{r}{d}\nsubseteq \BN{g}{1}{3}$. 
	\end{proof}
	
	\begin{remark}
		We note that the bielliptic curves above have $\Cliff(C)=2$. This follows directly from \Cref{lem:Cliff_1_cont}, which shows that if $C$ had Clifford index $\leq 1$, it would admit a $\g{1}{2}$ or a $\g{1}{3}$, and \Cref{thm:castelnuovo_severi} shows that bielliptic curves of genus $\geq 5$ do not. However, as bielliptic curves admit all linear series with $\gamma=2$, bielliptic curves do not distinguish Brill--Noether loci with $\gamma=2$.
	\end{remark}
	
	We note that there are always some containments in Clifford index $2$.
	
	\begin{prop}\label[prop]{prop:Cliff_2_containments_equalities}
		If $g\geq 11$ and $e\geq 4$, then
		\begin{enumerate}[label={\normalfont(\roman*)}]
			\item $\BN{g}{e-1}{2e}\subset \BN{g}{2}{6}$ and curves in $\BN{g}{e-1}{2e}$ are either hyperelliptic or bielliptic, and
			\item $\BN{g}{2}{6}\nsubseteq \BN{g}{e-1}{2e}$.
		\end{enumerate}
	\end{prop}
	\begin{proof}
		We first prove (i). Let $C\in \BN{g}{e-1}{2e}$. One can easily check that \Cref{thm:castelnuovo_bound} shows that the $\g{e-1}{2e}$ cannot be basepoint free and very ample. If the $\g{e-1}{2e}$ has a basepoint, then from \Cref{lem:Cliff_1_cont} we see that $C\in\BN{g}{e-1}{2e-1}=\BN{g}{1}{2}\subset\BN{g}{2}{6}$. Thus we may assume that the $\g{e-1}{2e}$ factors as \[C \stackrel{k:1}{\longrightarrow} D \stackrel{\g{e-1}{2e/k}}{\longrightarrow} \PP^{e-1}.\] If $k\geq 3$, then we note that $\frac{2e}{k}<e-1$, hence $D$ cannot admit such a $\g{e-1}{2e/k}$. Hence we have $k=2$. If $g(D)\geq 2$, then Riemann--Roch shows that the $\g{e-1}{e}$ is special, contradicting Clifford's theorem. Therefore $k=2$ and $g(D)\leq 1$, and $C$ is either hyperelliptic or bielliptic and $C\in \BN{g}{2}{6}$, as claimed.
		
		Statement (ii) follows from the fact that $\maxk(g,2,6)=3$, and using \Cref{eq:maxk_formula} one can readily verify that $\maxk(g,e-1,2e)=2$. Indeed, we note that $g\leq 2e+1$, as we assume that $d\leq g-1$. The case $g<2e+1$ follows easily, as then $\kappa(g,e-1,2e)=\floor{\frac{2e}{e-1}}=2$ for $e\geq 4$. If $g=2e+1$, we show that $\kappa(2e+1,e-1,2e)=2$. Suppose for contradiction that $\kappa(2e+1,e-1,2e)\geq 3$; we compute 
		\begin{align*}
			& \kappa(2e+1,e-1,2e) \geq 3\\[0.5ex]
			\iff & 2e+\floor{-2\sqrt{-\rho(2e+1,e-1,2e)}} \geq 3 \\[0.5ex]
			\iff & \floor{-2\sqrt{e^2-2e-1}} \geq 3 -2e \\[0.5ex]
			\implies & -2\sqrt{e^2-2e-1}  \geq 3-2e \\[0.5ex]
			\iff & \sqrt{e^2-2e-1} \leq \frac{2e-3}{2} \\[0.5ex]
			\iff & e^2-2e-1 \leq \frac{4e^2-12e+9}{4} \text{ as both sides are positive}\\[0.5ex]
			\iff & e\leq \frac{13}{4},
		\end{align*}
		which is a contradiction as $e\geq 4$. Thus $\kappa(g,e-1,2e)=2$. As $\kappa(g,2,6)=3$, (ii) follows from \Cref{thm:rho_k}.
	\end{proof}
	\begin{remark}\label[remark]{rmk:Cliff_2_equalities}
		In particular, for $g\geq 11$ and $e\geq 4$, the Brill--Noether loci $\BN{g}{e-1}{2e}$ are all equal.
	\end{remark}
	
	\begin{remark}\label[remark]{rmk:equalities_fixed_cliff}
		In the spirit of~\cite[2.56b]{Martens_Cliff_1}, it appears that for some range of $g,r,d$ and $\gamma$ sufficiently small, the Brill--Noether loci $\BN{g}{r}{d}$ with $d-2r=\gamma$ will either be equal or contained in Brill--Noether loci of lower Clifford index. For example, for $\gamma=3$, \cite[2.56b]{Martens_Cliff_1} shows that for $13\leq 2r+3=d\leq g-1$, the Brill--Noether loci $\BN{g}{r}{d}$ are all contained in loci of lower Clifford index, and hence are equal by \Cref{prop:Cliff_2_containments_equalities}.
	\end{remark}
	
	An interesting non-trivial containment is $\BN{11}{2}{6}= \BN{11}{3}{9}$, which persists in higher genus, and is not secant expected.
	
	\begin{prop}\label[prop]{prop:g26=g39}
		Let $g=11$, or $g \geq 13$, then $\BN{g}{2}{6}= \BN{g}{3}{9}$.
	\end{prop}
	\begin{proof}
		
		We first show $\BN{g}{2}{6}\subseteq \BN{g}{3}{9}$, this only assumes $g\geq 11$.
		Let $C\in \BN{g}{2}{6}$. The $\g{2}{6}$ cannot be very ample, else $C$ would be a smooth planar curve of genus $10$. Thus either the $\g{2}{6}$ has a basepoint, in which case $C\in\BN{g}{2}{5}=\BN{g}{1}{2}\subseteq \BN{g}{3}{9}$, or else the $\g{2}{6}$ factors as \[C \stackrel{k:1}{\longrightarrow} D \stackrel{\g{2}{6/k}}{\longrightarrow} \PP^2.\] Clearly $k=2$ or $k=3$, and \Cref{thm:lange_birat_dim} shows that $g(D)\leq 1$. However, the case $k=3$ and $g(D)=1$ does not occur, as elliptic curves do not admit a $\g{2}{2}$. In any case, $C$ is either hyperelliptic, and we see that $C\in\BN{g}{3}{9}$ as above; is bielliptic, hence $C\in\BN{g}{3}{8}\subseteq \BN{11}{3}{9}$; or is trigonal, whereby $C\in\BN{g}{1}{3}\subseteq \BN{11}{3}{9}$, as one can easily check using \Cref{eq:maxk_formula} that $\kappa(g,3,9)=3$ for $g\geq 11$.
		
		We now show $\BN{g}{3}{9}\subseteq \BN{g}{2}{6}$, assuming $g=11$ or $g\geq 13$.
		We see from \Cref{thm:castelnuovo_bound} that the $\g{3}{9}$ cannot be very ample and basepoint free if $g\geq 13$. For $g=11$, we note that there are no smooth curves of genus $11$ and degree $9$ in $\PP^3$, see~\cite[IV~Ex.~6.4]{Hartshorne_AG}. Suppose that the $\g{3}{9}$ factors through a lower genus curve as \[C\stackrel{3:1}{\longrightarrow} D \stackrel{\g{3}{3}}{\longrightarrow} \PP^3,\] hence $g(D)=0$, and $C\in\BN{g}{1}{3}\subseteq \BN{g}{3}{9}$, as above. Otherwise the $\g{3}{9}$ has a basepoint and $C\in \BN{g}{3}{8}$. If the $\g{3}{8}$, now, has a basepoint, then $C\in\BN{g}{3}{7}=\BN{g}{1}{2}\subset \BN{g}{3}{9}$. We are left with the case that the $\g{3}{8}$ factors as \[C \stackrel{k:1}{\longrightarrow} D \stackrel{\g{3}{8/k}}{\longrightarrow} \PP^3.\] Clearly $k=2,4$. However, if $k=4$, then $D$ admits a $\g{3}{2}$, which does not occur. Hence $k=2$. However, we cannot have $g(D)\geq 2$, as Riemann--Roch implies that the $\g{3}{4}$ is special, contradicting Clifford's theorem. Hence $g(D)\leq 1$, and $C$ is either hyperelliptic or bielliptic, and in both cases $C\in\BN{g}{2}{6}$.
	\end{proof}
	
	\begin{remark}
		In genus $12$, there are smooth genus $12$ curves of degree $9$ in $\PP^3$, they are Castelnuovo curves and lie on a quadric cone. From~\cite[III~Corollary~2.6]{ACGH}, these curves admit a $\g{1}{4}$, and we see from~\cite[Lemma~3.6]{Accola79_castelnuovo_curves} that these Caselnuovo curves have gonality $\geq 4$. We give additional non-containments of $\BN{12}{3}{9}$ and $\BN{12}{4}{11}$, both of which contain Castelnuovo curves, see \Cref{subsec:BN Strat genus 12}.
	\end{remark}
	
	We identify some additional non-containments of $\BN{g}{2}{6}$, and a corollary for Brill--Noether loci in Clifford index $3$.
	\begin{prop}
		Let $g\geq 11$, then we have
		\begin{enumerate}[label={\normalfont(\roman*)}]
			\item $\BN{g}{2}{d}\nsubseteq \BN{g}{1}{3}$ for $d\geq 6$,
			\item $\BN{g}{2}{6}\nsubseteq \BN{g}{s}{e}$ for $11\leq 2s+3=e\leq g-1$, and
			\item $\BN{g}{3}{9}\nsubseteq \BN{g}{s}{e}$ for $11\leq 2s+3=e\leq g-1$.
		\end{enumerate} 
	\end{prop}
	\begin{proof}
		For $d=6$, \Cref{lem:Cliff_2_bielliptic} shows that $\BN{g}{2}{6}\nsubseteq \BN{g}{1}{3}$, and the general statement (i) follows immediately from the trivial containments. 
		
		To prove (ii), one can readily check that since $f\leq g-1$, $\kappa(g,s,f)=\floor{\frac{2s+3}{s}}$ by \Cref{eq:maxk_formula}, and as $s\geq 4$, we have $\kappa(g,s,f)=2$. As $\kappa(g,2,6)=3$, (iii) follows from \Cref{thm:rho_k}.
		
		The last statement follows for $g\neq 12$ from the fact that $\BN{g}{2}{6}=\BN{g}{3}{9}$, by \Cref{prop:g26=g39}, but (ii) shows that $\BN{g}{2}{6}\nsubseteq \BN{g}{s}{e}$ for the range given; and for $g=12$, one can check that $\kappa(12,3,9)=3$ and $\kappa(12,4,11)=2$, hence $\BN{12}{3}{9}\nsubseteq\BN{12}{4}{11}$ by \Cref{thm:rho_k}, as claimed.
	\end{proof}

	\subsection{Genus 11}\label{subsec: BN Strat genus 11}
	
	 Using the methods for genus $\leq 10$, and the additional containments in low Clifford index, one can identify almost all containments and noncontainments among Brill--Noether loci in genus $11$, see \Cref{fig:g_11_cont}. However, the containments $\BN{11}{3}{10}\subset \BN{11}{1}{6}$ and $\BN{11}{2}{7}\subset \BN{11}{3}{10}$ are slightly different, and are explained here.  
	
	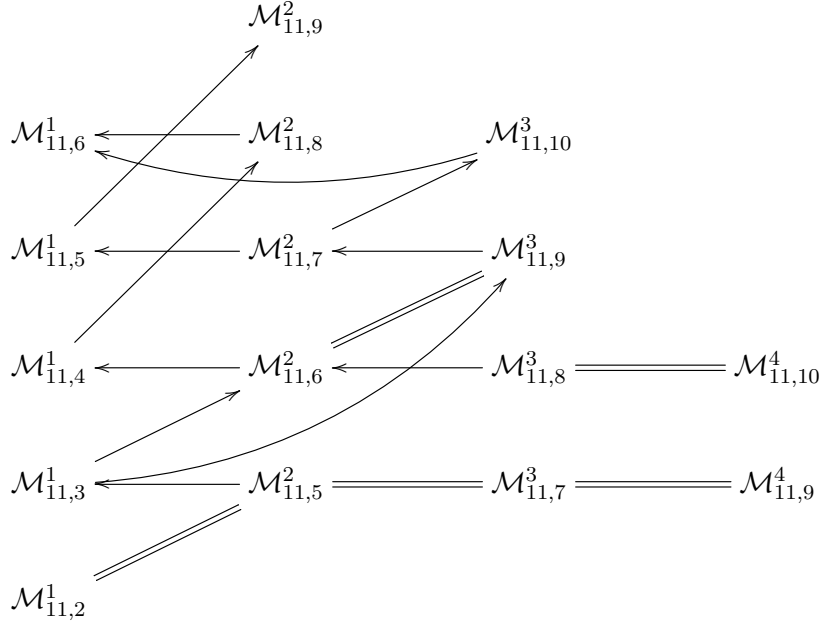
\begin{figure}[H]
		\[
		\xymatrix{
			& & \BN{11}{2}{9} \\
			\BN{11}{1}{6} & & \BN{11}{2}{8} \ar[ll] & & \BN{11}{3}{10} \ar@/^1.5pc/ [llll] \\
			\BN{11}{1}{5} \ar[uurr] & & \BN{11}{2}{7} \ar[urr] \ar[ll] & & \BN{11}{3}{9} \ar[ll] \\
			\BN{11}{1}{4} \ar[uurr] & & \BN{11}{2}{6} \ar[ll] \ar@{=}[urr]  & & \BN{11}{3}{8} \ar[ll]  & & \BN{11}{4}{10} \ar@{=}[ll]  \\
			\BN{11}{1}{3} \ar[urr] \ar@/_2pc/[uurrrr] & &  \BN{11}{2}{5} \ar@{=}[rr] \ar[ll] & & \BN{11}{3}{7} \ar@{=}[rr] & & \BN{11}{4}{9}\\
			\BN{11}{1}{2} \ar@{=}[urr] 
		}\]
		\caption{Non-trivial containments of Brill--Noether loci in genus $11$.}\label{fig:g_11_cont}
	\end{figure}

	The first containment is via a classical formula for $4$-secant lines to curves in $\PP^3$.
	
	\begin{prop}
		There is a containment $\BN{11}{3}{10}\subset \BN{11}{1}{6}$.
	\end{prop}
	\begin{proof}
		Suppose $C\in\BN{11}{3}{10}$. If the $\g{3}{10}$ is not basepoint free, or very ample, one can easily show that $C\in\BN{11}{1}{6}$. So we are left with the case that the $\g{3}{10}$ is very ample, and $C\subset \PP^3$ is a smooth genus $g=11$ curve of degree $d=10$. From classical formulas of Cayley, see for example~\cite[Example following VII~Proposition~4.2]{ACGH}, \[\#\{\text{$4$-secant lines to $C$}\}=\frac{(d-2)(d-3)^2(d-4)}{12}-\frac{g(d^2-7d+13-g)}{2}=20,\] whereby $C$ admits a $4$-secant line, and projecting from such a line gives a $\g{1}{6}$.
	\end{proof}
	
	The second containment is $\BN{11}{2}{7}\subset \BN{11}{3}{10}$, which is obtained via an explicit construction of a linear series on a nodal plane curve, for which the author is indebted to Isabel Vogt. 
	
	\begin{remark}\label[remark]{rmk:k3exp_secunex}
		We note that the containment $\BN{11}{2}{7}\subset \BN{11}{3}{10}$ is secant-unexpected, as \[\operatorname{exp}\dim V^{2}_{3}(\g{5}{13})=-1<0,\] using that the Serre dual of a $\g{2}{7}$ is a $\g{5}{13}$. However, the containment is K3-expected, as on a K3 surface with $\Pic(S)=\Lambda^2_{11,7}$, the bundle $L\oplus L \oplus E_{D,\g{1}{2}}$ is the Lazarsfeld--Mukai bundle of a $\g{3}{10}$ on a smooth irreducible $C\in|H|$, where $D\in|L|$ is a smooth irreducible genus $2$.
	\end{remark}

	\begin{prop}\label[prop]{prop:Vogt_g3_10_construction}
		There is a containment $\BN{11}{2}{7}\subset \BN{11}{3}{10}$.
	\end{prop}
	\begin{proof}
		Let $C\in\BN{11}{2}{7}$. If the $\g{2}{7}$ is not basepoint free, then $C\in\BN{11}{2}{6}=\BN{11}{3}{9}\subset\BN{11}{3}{10}$. Hence we are left with the case that the $\g{2}{7}$ is basepoint free and, to reduce notation, we let $C$ be a plane septic curve with $4$ nodes. We note that the nodes can be taken to be general points in $\PP^2$. Similar to geometric Riemann--Roch, to obtain a $\g{3}{10}$ on $C$, it suffices to find $10$ points in $\PP^2$ such that the projective space of quartic curves through the nodes and the $10$ chosen points has dimension $3$. Indeed, the system of such quartics will intersect $C$ in $10$ additional points, and the dual projective space will give a $\g{3}{10}$.
		
		We take the system of quartics to be of the form $L_1L_2Q$, where we choose $7$ co-linear points not on $C$, and $L_1$ is a line through those $7$ points, $L_2$ is the line through two nodes of $C$ and we choose $3$ additional points on $L_2$ not on $C$, and $Q$ is a conic through the remaining two nodes of $C$. The space of such conics $Q$ has dimension $\binom{4}{2}-1-1=4$, as the nodes of $C$ are general points. Hence the space of quadrics of the form $L_1 L_2 Q$ has projective dimension $3$, as desired.
	\end{proof}

	In total, we have identified the relative positions of Brill--Noether loci in genus $11$.
	
	\begin{theorem}\label[theorem]{thm:BNloci_genus_11_summary}
		All non-trivial containments among Brill--Noether loci in genus $11$ are obtained via trivial containments and the arrows appearing in \Cref{fig:g_11_cont}.
	\end{theorem}

	 \subsection{Genus 12}\label{subsec:BN Strat genus 12}
	 
	 It becomes more difficult to distinguish all Brill--Noether loci in higher genus. Using the results used in genus $\leq 11$, many containment and non-containments can be shown, and we fill in the reaming gaps with ad hoc constructions using Castelnuovo curves and linear systems of curves through nodal plane curves, as in genus $11$. The relative positions of Brill--Noether loci in genus $12$ are summarized in \Cref{fig:g_12_cont}.

	 \begin{figure}[H]
	 	\[
	 	\xymatrix@R+1.1pc{	 		
	 		 & & \BN{12}{2}{9} & & \BN{12}{3}{11} \\ 
	 		\BN{12}{1}{6} & & \BN{12}{2}{8} & & \BN{12}{3}{10} \\
	 		\BN{12}{1}{5} & & \BN{12}{2}{7} \ar[ll] \ar[urr] & & \BN{12}{3}{9} \ar@/_0.5pc/[dllll] & & \BN{12}{4}{11} \ar@/^1.7pc/[llll] \ar[dllllll] \\
	 		\BN{12}{1}{4} \ar[uurr] \ar[uuurrrr] & & \BN{12}{2}{6} \ar[urr] \ar@/_0.9pc/[urrrr] \ar[ll] & & \BN{12}{3}{8} \ar[ll] \ar@{=}[rr] & & \BN{12}{4}{10} \\
	 		\BN{12}{1}{3} \ar[urr] \ar@/_1.5pc/[uurrrr] & & \BN{12}{2}{5} \ar@{=}[rr] & & \BN{12}{3}{7} \ar@{=}[rr] \ar[u] & & \BN{12}{4}{9} \ar@{=}[rr] & & \BN{12}{5}{11}\\
	 		\BN{12}{1}{2} \ar@{=}[urr] 
	 	}\]
	 	\caption{Non-trivial containments of Brill--Noether loci in genus $12$.}\label{fig:g_12_cont}
	 \end{figure}

	 \begin{remark}
	 	It is straightforward to show that the curves in $\BN{12}{2}{6}$ are either trigonal or hyperelliptic.
	 \end{remark} 
	 
	 We give a few results using Castelnuovo curves, following~\cite[Corollary~5.2, Example~5.3]{Chiantini_Ciliberto_halphen_99}, which do not follow using techniques used in genus $\leq 11$.
	 
	 \begin{prop}\label[prop]{prop:genus_12_containments}
	 	We have the following containments and non-containments
	 	\begin{enumerate}[label={\normalfont(\roman*)}]
	 		\item $\BN{12}{3}{9}\subset \BN{12}{1}{4}$ and $\BN{12}{3}{9}\nsubseteq\BN{12}{1}{3}$,
	 		\item $\BN{12}{4}{11}\subset \BN{12}{1}{4}$ and $\BN{12}{4}{11}\nsubseteq\BN{12}{1}{3}$,
	 		\item $\BN{12}{3}{9}\nsubseteq \BN{12}{2}{7}$,
	 		\item $\BN{12}{4}{11}\subseteq \BN{12}{2}{7}$ and $\BN{12}{4}{11}\nsubseteq \BN{12}{2}{6}$,
	 		\item $\BN{12}{4}{11}\nsubseteq \BN{12}{3}{9}$, and
	 		\item $\BN{12}{2}{7}\nsubseteq\BN{12}{3}{9}$.
	 	\end{enumerate}
	 \end{prop}
	 \begin{proof}
	 	For (i)--(v), we start with $C\in\BN{12}{3}{9}$ or $\in\BN{12}{4}{11}$, and we may assume the $\g{3}{9}$ or $\g{4}{11}$ is basepoint free and very ample, as otherwise the containments follow from other containments or trivial containments. We consider Castelnuovo curves in $\PP^3$ and $\PP^4$.
	 	
	 	The proof of (i) and (ii) follows from~\cite[III~Corollary~2.6]{ACGH} which shows that Castelnuovo curves in $\PP^3$ and $\PP^4$ have a $\g{1}{4}$, and~\cite[Lemma~3.6]{Accola79_castelnuovo_curves} shows they have gonality $4$.
	 	
	 	The non-containments in (iii)--(v) follow from~\cite[Corollary~5.2]{Chiantini_Ciliberto_halphen_99} (as remarked in~\cite[Example~5.3]{Chiantini_Ciliberto_halphen_99}, the corollary applies) which gives for a Castelnuovo curve $C$ of genus $g$ in $\PP^r$ with $r=3,4$ of degree $d$, a $\g{N}{\delta}$ with $2\leq N\leq r-1$ on $C$ satisfies \[\delta \geq g-\epsilon\rho +N-\rho - (\rho^2-\rho)\frac{r-1}{2},\] where $\rho\geq0$ and $\epsilon$ are defined by \[d-N-2=\rho(r-1)+\epsilon,\ \ 0\leq\epsilon <r-1.\]
	 	
	 	To prove (iii), consider $C\in\BN{12}{3}{9}$ a smooth Castelnuovo curve. For a $\g{2}{\delta}$ on $C$, we have $\rho=2$ and $\epsilon=1$, whereby~\cite[Corolalry~5.2]{Chiantini_Ciliberto_halphen_99} shows that $\delta\geq8$, and so $C\notin \BN{12}{2}{7}$, as claimed.	
	 	
	 	For (iv), consider $C\in\BN{12}{4}{11}$ a Castelnuovo curve. For a $\g{2}{\delta}$ on $C$, we have $\rho=2$ and $\epsilon=1$, whereby~\cite[Corolalry~5.2]{Chiantini_Ciliberto_halphen_99} shows that $\delta\geq 7$, hence $C\notin \BN{12}{2}{6}$, as claimed. 	
	 	
	 	For (v), consider $C\in\BN{12}{4}{11}$ a Castelnuovo curve. For a $\g{3}{\delta}$ on $C$, we have $\rho=2$ and $\epsilon=0$, whereby~\cite[Corollary~5.2]{Chiantini_Ciliberto_halphen_99} shows that $\delta\geq 10$, thus $\BN{12}{4}{11}\nsubseteq\BN{12}{3}{9}$.
	 	
	 	To show the containment in (iv), that $\BN{12}{4}{11}\subset \BN{11}{2}{7}$, we again can assume the $\g{4}{11}$ is very ample, and $C\in\BN{12}{4}{11}$ is a Castelnuovo curve on the cubic scroll in $\PP^4$, which is isomorphic to the Hirzebruch surface $\F_1$. The curve $C$ is a general curve on the Hirzebruch surface meeting the unique curve of self-intersection $(-1)$ in degree $3$, and has a $\g{1}{4}$. From~\cite[Theorem~1.5]{larson_2024brillnoethertheorysmoothcurves} and~\cite[Theorem~1.9]{larson_2024brillnoethertheorysmoothcurves}, we see that the double splitting locus $U^{\vec{e}, \vec{f}}(C)$ for a line bundle $L$ of type $\g{2}{7}$ given by \[\vec{e}=(-5,-4,0,1),\ \vec{f}=(-5,-1,0,1)\] is non-empty, and thus the general curve with a very ample $\g{4}{11}$ admits a $\g{2}{7}$.
	 	
	 	Finally, to show (vi), we note that by \Cref{thm:coppens_plane_gon}, a plane curve of degree $7$ and genus $12$ with $3$ nodes has no $\g{1}{4}$. And now (vi) follows, as by (i) we have the containment $\BN{12}{3}{9}\subset \BN{12}{1}{4}$.
	 \end{proof}
	 
		This leaves only the potential containments of $\BN{12}{2}{7}$ into $\BN{12}{3}{10}$ or $\BN{12}{3}{11}$, for which we give a construction in \Cref{prop:genus_12_g_2_7_has_g_3_10} similar to \Cref{prop:Vogt_g3_10_construction}. An understanding the Brill--Noether theory of nodal plane curves with a small number of nodes may also address similar containments. We first note the K3-expected containments of $\BN{12}{2}{7}$.
		
		\begin{example}\label[example]{ex:LM_bund_on_12_2_7}
			We note that the containments $\BN{12}{2}{7}\subset \BN{12}{3}{10}\subset \BN{12}{3}{11}$ is K3-expected. Indeed, on a K3 surface with $\Pic(S)=\Lambda^2_{12,7}$, there are admissible assignments for Lazarsfeld--Mukai bundles of a $\g{3}{10}$ and a $\g{3}{11}$ on a smooth irreducible $C\in|H|$. There are two admissible assignments for a Lazarsfeld--Mukai bundle of a $\g{3}{10}$ on $C\in|H|$, giving bundles \[(H-2L) \oplus L \oplus E_{L,\g{1}{2}}, \ \ L \oplus L \oplus E_{H-2L,\g{1}{2}},\] if $|H-2L|$ admitted smooth curves. However, we note that $\Lambda^2_{12,7}$ has many $(-2)$-curves, namely $H-3L$, $-H+4L$, $2H-5L$, and $-2H+9L$, and we see that $H-2L$ has a fixed component $H-3L$.
			
			Furthermore, given $D_1\in |H-L|$ and $D_2\in |L|$ be smooth curves, then $E_{D_1,\g{1}{4}}\oplus E_{L,\g{1}{2}}$ is the Lazarsfeld--Mukai bundle of a $\g{3}{11}$ on $C$.
			
			We also note that the containment $\BN{12}{2}{7}\subset \BN{12}{3}{10}$ is also predicted by admissible fillings of tableaux; and also by considering more special curves $C\in\BN{12}{2}{7}$ whose image in $\PP^2$ has a triple point which, after blowing up the triple point, lie on the Hirzebruch surface $\mathbb{F}_1$ and admit a $\g{3}{10}$, as can be checked from \cite[Theorem~1.5]{larson_2024brillnoethertheorysmoothcurves}.
		\end{example}
		
		We thank Dave Jensen for the following containment.
		
		\begin{prop}\label[proposition]{prop:genus_12_g_2_7_has_g_3_10}
			There is a containment $\BN{12}{2}{7}\subset \BN{12}{3}{10}$.
		\end{prop}
		\begin{proof}
			Let $C\in\BN{12}{2}{7}$. We first note that we may assume that the $\g{2}{7}$ is basepoint free, as otherwise $C\in\BN{12}{2}{6}\subset \BN{12}{3}{9}\subset \BN{12}{3}{10}$ by \Cref{prop:genus_12_containments}. Thus, as $7$ is prime, we may assume that the $\g{2}{7}$ maps $C$ birationally onto its image in $\PP^2$, and that $C$ is general in the image of the Severi variety, and we may assume that the image of $C$ under the $\g{2}{7}$ is a plane curve with $3=\binom{7-1}{2}-12$ nodes, which may be taken to be in general position. Taking the $\binom{4}{2}-1-1=4$ dimensional linear system of conics through $2$ of the nodes gives a $\g{3}{10}$, as each conic intersects the image of $C$ in $10$ additional points.
		\end{proof}
		
		\begin{theorem}\label[theorem]{thm:BNloci_genus_12_summary}
			All non-trivial containments among Brill--Noether loci in genus $12$ are obtained via trivial containments and the arrows appearing in \Cref{fig:g_12_cont}.
		\end{theorem}
	
	\subsection{Genus \texorpdfstring{$\geq13$}{}}\label{subsec: BN Strat genus geq 13}	
	
	In general, one may provide a clear picture of the relative positions of Brill--Noether loci of not too large codimension, using some ad hoc constructions as in genus $11$ and $12$, where possible. 
	
	For example, using the results of~\cite{ahk_2024}, in some range, $\Pic(S)=\Lambda^r_{g,d}$ has a unique \emph{flexible decomposition} of $H$, and K3 surfaces suffice to identify non-containments. There may be a larger range still, where flexible decompositions are not unique, but an analysis of admissible assignments and other methods may be sufficient, which we pursue in forthcoming work. In particular, as in \Cref{ques:range_no_nontriv_cont}, it would be of interest to identify a range where there are no non-trivial containments among Brill--Noether loci, though this range would have to occur in large genus where $\dmax(g,r)-2r>\floor{\frac{g-1}{2}}$. 
	
	In low Clifford index, there are also a range of new containments, as linear series are forced to have basepoints, as in \Cref{prop:Cliff_2_containments_equalities}, and in genus $\geq 13$ we have the following equalities.
	
	\begin{prop}\label[prop]{prop:g38_equalities}
		Let $g\geq 13$, then $\BN{g}{4}{11}= \BN{g}{4}{10}=\BN{g}{3}{8}$.
	\end{prop}
	\begin{proof}
		From \Cref{thm:castelnuovo_bound}, a $\g{4}{11}$ on a curve of genus $\geq 13$ must have basepoints or factor through a lower genus curve. However, as $11$ is prime, it cannot factor, and thus $\BN{g}{4}{11}\subseteq \BN{g}{4}{10}$, the reverse containment is a trivial containment. Finally, $\BN{g}{4}{10}=\BN{g}{3}{8}$ is \Cref{prop:Cliff_2_containments_equalities} and \Cref{rmk:Cliff_2_equalities}.
	\end{proof}
	
	\begin{remark}
		Note that for $g\geq13$, we have $\BN{g}{3}{9}\nsubseteq\BN{g}{3}{8}$, as the former contains trigonal curves, while the latter contains only hyperelliptic and bielliptic curves.
	\end{remark}
	
	In general, as mentioned in \Cref{rmk:equalities_fixed_cliff}, it would be interesting to determine the ranges where Brill--Noether loci are equal.

	\section{Conjectures on Brill--Noether loci}\label{sec:Conjectures}
	
	In general, due to the existence of multiple components of various dimensions, a complete picture of the relative positions of Brill--Noether loci in higher genus remains elusive. We motivate a few conjectures, using curves on K3 surfaces to provide numerical expectations, and give a few explicit examples in genus $100$ to illustrate the numerical intuition.
	
	\subsection{Irreducibility}\label{subsec:irred_conjs} 
	
	When $\BN{g}{r}{d}$ has multiple components, it no longer makes sense to ask about the behavior of a general curve with a $\g{r}{d}$. It would be appealing to identify a distinguished component of Brill--Noether loci that is well-behaved and whose Brill--Noether theory can be described. 
	
	There is no general expectation when Brill--Noether loci should be irreducible, though certainly there are many examples of reducible Brill--Noether loci. As pointed out in~\cite[Appendix~A.1]{pflueger_legos}, one way to find reducible Brill--Noether loci is from containments $\BN{g}{1}{k}\subseteq \BN{g}{r}{d}$. In~\cite[Corollary~4.7]{h_tib_theta}, these containments in fact show that $\BN{g}{2}{d}$ has at least two components for $\frac{g}{3}+3\leq d < \floor{\frac{g+4}{2}}$ and $g\geq 9$. It is natural to ask if this is perhaps the first origin of another component of Brill--Noether loci.
	
	\begin{question}\label[question]{ques:irred_kappa_bound}
		Is $\BN{g}{r}{d}$ irreducible when $0>\rho(g,r,d)>\rho(g,1,\kappa(g,r,d))$?
	\end{question}
	
	Similarly, one is tempted to ask whether Brill--Noether are irreducible when $\rho$ is not too negative.
	
	\subsection{Gonality of curves in \texorpdfstring{$\BN{g}{r}{d}$}{}}
	
	Turning the question above on its head, one may ask: Given a curve $C\in \BN{g}{r}{d}$, what is the gonality of $C$? When $\BN{g}{r}{d}$ is reducible, one may ask about the gonality of general curves of each component. In terms of the relative positions of Brill--Noether loci, it is natural to consider the smallest $k$ such that all curves with a $\g{r}{d}$ admit a $\g{1}{k}$, which is also the maximum gonality of the general curve in each component.
	
	\begin{defn}
		For a Brill--Noether locus $\BN{g}{r}{d}$, we define \[K(g,r,d)\coloneqq \min\{ k \ \mid \ \BN{g}{r}{d}\subseteq \BN{g}{1}{k}\}.\]
	\end{defn}
	
	Clearly, we have $\kappa(g,r,d)\leq  K(g,r,d)\leq \floor{\frac{g+3}{2}}$. Moreover, for plane curves, the results of~\cite{Coppens_91,Copopens_Kato_gon_nodal_plane_curves} show that $K(g,2,d)=d-2$ in many cases.
	
	Again, one may try to use curves on K3 surfaces to provide some expectation for $K(g,r,d)$. Let $S$ be a K3 surface with $\Pic(S)=\Lambda^r_{g,d}$ and $C\in|H|$ a smooth irreducible curve. When $g,r,d$ satisfy certain numerical constraints, \cite[Theorem~3]{Farkas_2001} shows that we have  \[\gon(C)=\min\{d-2r+2,~\floor{(g+3)/2}\}.\] In some range, it is perhaps expected that $K(g,r,d)=\min \{d-2r+2,~\floor{(g+3)/2}\}$. However, already in genus $11$ and $12$ this does not hold. As observed in \Cref{prop:g26=g39} and \Cref{prop:genus_12_containments}, $\BN{11}{3}{9}=\BN{11}{2}{6}\subset \BN{11}{1}{4}$, $\BN{12}{3}{9}\subset \BN{12}{1}{4}$, and $\BN{12}{4}{11}\subset \BN{12}{1}{4}$, and for higher genus one has \Cref{prop:g26=g39} and \Cref{prop:g38_equalities}.
 	 	
 	This expectation does not hold on K3 surfaces in general. In particular, when $S$ has elliptic curves, the gonality of smooth curves in $|H|$ may be much lower than $\min \{d-2r+2, \floor{(g+3)/2}\}$. We give two examples. In both examples, the numerical conditions on $g,r,d$ of~\cite[Theorem~3]{Farkas_2001} are not satisfied.
 	
 	\begin{example}[Gonality of curves on K3s with $\Pic(S)=\Lambda^r_{g,d}$] 	
 	Let $\Pic(S)=\Lambda^{10}_{100,60}$. We note that there are no $(-2)$-curves, but there are $(0)$-curves, namely $n(H-3L)$ and $n(-3H+11L)$. We show that for a smooth irreducible curve $C\in|H|$, $\gon(C)=18<42=\min\{d-2r+2, \floor{(g+3)/2}\}$. We consider all admissible assignments of a Lazarsfeld--Mukai bundle of a $\g{1}{e}$ with $\rho(g,1,e)<0$ on $C$. There is a Lazarsfeld--Mukai bundle of the form $(-H+6L)\oplus(2H-6L)$, which is the Lazarsfeld--Mukai bundle of a $\g{1}{36}$, and $H.(2H-6L)=36$, so $2H-6L$ is a lift of the $\g{1}{36}$. In this case we have $\gon(C)=36<60-2\cdot10+2=42$. There is also another Lazarsfeld--Mukai bundle of the form $(2H-4L)\oplus (-H+4L)$, also the Lazarsfeld--Mukai bundle of a $\g{1}{36}$, but $H.(-H+4L)=42$, and the $\g{1}{36}$ is just contained in the $\g{1}{42}$ ($(-H+4L)$ is a Donagi--Morrison lift of the $\g{1}{36}$). Finally, there is another Lazarsfeld--Mukai bundle of the form $(3L)\oplus (H-3L)$, which is the Lazarsfeld--Mukai bundle of a $\g{1}{18}$, and we note that $\calO_C(H-3L)$ is a $\g{1}{18}$. Any other admissible assignment gives a larger lower bound on $c_2(E_{C,\g{1}{e}})$, thus $\gon(C)=18<42$, as claimed.
 	
 	On the other hand, if $\Pic(S)=\Lambda^{10}_{100,61}$, then for a smooth irreducible curve $C\in|H|$, we have $\gon(C)=43=\min\{d-2r+2, \floor{(g+3)/2}\}$. We note that $\Lambda^{10}_{100,61}$ has a $(-2)$-curve ($-H+4L$) and no $(0)$-curves. Indeed, $(H-L)\oplus L$ is a Lazarsfeld--Mukai bundle of a $\g{1}{43}$, and an explicit computation of the admissible assignments for a Lazarsfeld--Mukai bundle shows that all other admissible assignments give a higher $c_2$ bound. Thus even though $g,r,d$ do not satisfy the numerical constraints of~\cite[Theorem~3]{Farkas_2001}, $\gon(C)=43$. 
 \end{example}
 
 With additional knowledge of the curves in a given component of $\BN{g}{r}{d}$, one may provide additional bounds on $K(g,r,d)$, as in~\cite[Lemma~5.5]{h_tib_theta}, which provides bounds on the gonality of curves in certain components of $\BN{g}{r}{g-1}$ for $g=\binom{r+2}{2}$. Namely, $\BN{g}{r}{g-1}$ has at least two components, one of curves where the $\g{r}{g-1}$ is a theta characteristic and another where it is not, and a lower bound on the gonality of curves in each component is provided by~\cite[Lemma~5.5]{h_tib_theta}. Specifically, when $g=\binom{r+2}{2}$, we have $K(g,r,d)\geq 2r$. In particular, we have the following two examples where $K(g,r,d)\neq \min\{d-2r+2,~\floor{(g+3)/2} \}$.
 
 \begin{prop}
 	For $r=3$ or $r=4$ and $g=\binom{r+2}{2}$, we have \[K(g,r,g-1)\geq 2r> \min\{g-1-2r+2,\floor{(g+3)/2}\}.\]
 \end{prop}
 \begin{proof}
 	From~\cite[Lemma~5.5]{h_tib_theta}, we have $K(g,r,g-1)\geq 2r$, and one readily verify that \[d-2r+2=\frac{r(r-1)}{2} +2<2r, \text{ and } g-1-2r+2<\floor{(g+3)/2}.\qedhere\]
 \end{proof}
 
 However, examples where $K(g,r,d)\neq \min \{d-2r+2, \floor{(g+3)/2}\}$ seem fairly rare, occurring when $\rho(g,r,d)$ is very negative and $\BN{g}{r}{d}$ is far from maximal. Hence it is natural to expect that $K(g,r,d)=\min\{d-2r+2, \floor{(g+3)/2}\}$ in some range.

\subsection{Conjectures on distinguishing Brill--Noether loci in large genus via secant varieties}\label{subsec:conts_secants}

We turn now to identifying a rough range where we expect to have no non-trivial containments of Brill--Noether loci. We start be identifying secant expected containments.
 
 	\begin{theorem}
 		Let $r\geq3$ be odd. If $e\geq d-2r+2+\floor{\frac{r+3}{2}}$, then $\BN{g}{r}{d}\subseteq \BN{g}{2}{e}$.
 	\end{theorem}
 	\begin{proof}
 		Suppose $C\in \BN{g}{r}{d}$. A $\g{2}{e}$ would be obtained from a divisor of degree $(d-e)$ imposing $(r-3)$ independent conditions on the $\g{r}{d}$. The expected dimension of such ``secant divisors" is \[(d-e)-(d-e-(r-3)-1)(r-(r-3))\geq-r+2\floor{\frac{r+1}{2}}=\begin{cases}
 		0 & \text{ if $r$ is even}\\
 		1 & \text{ if $r$ is odd.}
 		\end{cases}\] As $r$ is odd, there exists such a divisor, giving a $\g{2}{e}$ on $C$, as claimed.
 	\end{proof}
 	
 	\begin{remark}
 		If $r$ is even, then there may still be an effective ``secant" divisor of degree $(d-e)$ imposing  $(r-3)$ conditions. However, as the virtual class of such ``secant" divisors becomes quite complicated, it is not immediate.
 	\end{remark}
 	
 	It is natural to ask whether $\BN{g}{r}{d}$ admits no containments on $\BN{g}{2}{e}$ outside the range of secant expected containments.
	 
	 \begin{conj}\label[conj]{conj:secant_noncont_s_eq_2}
	 	For $g$ sufficiently large, if $r\geq 3$,$d,e\leq g-1$, and $e<d-2r+2+\floor{\frac{r+3}{2}}$, then $\BN{g}{r}{d}\nsubseteq \BN{g}{2}{e}$.
	 \end{conj}
	 
	 \begin{remark}
	 	Computations of admissible assignments on K3 surfaces with $\Pic(S)=\Lambda^r_{g,d}$ for large genus support this conjecture. In~\cite[Theorem~1.4]{Lelli_Chiesa_2013}, Lelli-Chiesa uses K3 surfaces of Picard rank $2$ to verify this claim, with some mild restrictions on $g,r,d$ to avoid $(-2)$-curves and to ensure nice K3 surfaces exist, for $e<d-2r+5$ (which is equivalent to $\gamma(\g{2}{e})\leq \gamma(\g{r}{d})$).
	 \end{remark}
	 
	 More generally, we identify some range where there are secant expected containments.
	 
	 \begin{prop}\label[prop]{prop:secant_containments_general}
	 	Let $\BN{g}{r}{d}$ and $\BN{g}{s}{e}$ be Brill--Noether loci such that $r\geq s+1 \geq 2$. Suppose that $e\geq~d-2r+2+s+\floor{\frac{r-1}{2}}+k$ for some non-negative integer $k$. If $r+\floor{\frac{-r+1}{2}}s+ks>0$, then $\BN{g}{r}{d}\subseteq \BN{g}{s}{e}$. 
	 \end{prop}
	 \begin{remark}
	 	The assumption that $r+\floor{\frac{-r+1}{2}}s+ks>0$ forces $s$ to be fairly small (depending on $k$).
	 \end{remark}
	 \begin{proof}
	 	Let $\in\BN{g}{r}{d}$. Let $V=V^{r-s}_{d-e}(\g{r}{d})$ be the divisors of degree $(d-e)$ imposing $(r-s-1)$ independent conditions on the $\g{r}{d}$. 
	 	If $V\neq \emptyset$, then $C\in\BN{g}{s}{e}$. We recall that 
	 	\begin{align*}
	 		\dim V& \geq\operatorname{exp}\dim V\\[0.5ex]
	 		&=(d-e)-(d-e-(r-s-1)-1)(r-(r-s-1)) \\[0.5ex]
	 		&=r-s-(d-e-r+s)s.
	 	\end{align*} 
	 	It remains to show that $r-s-(d-e-r+s)s>0$.
	 	
	 	As $e-d\geq s+2-2r+\floor{\frac{r-1}{2}}+k$, we have 
	 	\begin{align*}
	 		r-s-(d-e-r+s)s &\geq r-s +\left(s+2-2r+\floor{\frac{r-1}{2}}+k+r-s \right)s\\[0.5ex]
	 		&=r+\floor{\frac{-r+1}{2}}s+ks.
	 	\end{align*} By assumption, $r+\floor{\frac{-r+1}{2}}s+ks>0$, whereby $V=V^{r-s}_{d-e}(\g{r}{d})\neq\emptyset$ and we obtain a $\g{s}{e}$, as claimed.
	 \end{proof}
	 
	 We note that similar results may hold with less restrictive assumptions on $r$ and $s$. Outside of the range where there are secant expected containments, it is natural to ask the following question.
	 
	 \begin{conj}\label[conj]{conj:general_non_cont}
	 	For $g$ sufficiently large, $d,e\leq g-1$, if $r\geq s+1\geq 3$ and $e< d-2r+2+s+\floor{\frac{r-1}{2}}$, then $\BN{g}{r}{d}\nsubseteq \BN{g}{s}{e}$.
	 \end{conj}
	 
	 Computations in large genus on K3 surfaces support this conjecture. We note that in more precise ranges of $r$ and $s$, it may be possible to obtain stronger results.

	In the range $s\geq r+1$, there may be additional non-trivial containments which are not secant expected. One source of examples of such non-trivial containments are the containments arising from the refined Brill--Noether theory for curves of fixed gonality of the form $\BN{g}{1}{k}\subseteq \BN{g}{r}{d}$.
 
	\begin{question}
		Identify additional non-trivial non-secant expected containments of Brill--Noether loci.
	\end{question}
 
	Nevertheless, when there are no secant expected containments, it may be the case that there are in fact no non-trivial containments among Brill--Noether loci. It would be interesting to identify a range where Brill--Noether loci admit no non-trivial containments.
 
	\begin{question}\label[question]{ques:range_no_nontriv_cont}
		For sufficiently large $g$, is there a range where there are no non-trivial containments of Brill--Noether loci? How does this range depend on $g$? Does it align with the range predicted by K3 surfaces or secant expected containments?
	\end{question}
	
	Roughly, is there a range of $r$ and $d$, depending on $g$, where the only containments are secant-expected; and further, outside of this range (for larger $d$, perhaps), are there no containments among Brill--Noether loci? For instance, expected maximal Brill--Noether loci admit no containments for large genus, see~\cite{ahk_2024}.
 
\subsection{On non-containments of the form \texorpdfstring{$\BN{g}{r}{d}\nsubseteq \BN{g}{s}{e}$}{} for \texorpdfstring{$s>r$}{} via K3 surfaces}\label{subsec:non_conts_no_H-L_subline}

	Showing (non-)containments of the form $\BN{g}{r}{d}\nsubseteq \BN{g}{s}{e}$ for $s\geq r+1$ appears to be more difficult, as there are many admissible assignments when $d$ is small. Here we explain the general trends coming from admissible assignments, and formulate an expectation (from K3 surfaces) for the behavior in general.
	
	On a K3 surface with $\Pic(S)=\Lambda^r_{g,d}$, empirical evidence suggests that, in many cases, the lowest bound on $c_2$ for an unstable Lazarsfeld--Mukai bundle $E_{C,\g{s}{e}}$ comes from admissible assignment of the form $(H-L)\subset E_{C,\g{s}{e}}$, and other admissible assignments allow larger values of $c_2$. However, in the range $s\geq r+1$, $E_{C,\g{s}{e}}$ cannot have a terminal filtration of type $1\subset s+1$ with $c_1(E_1)=H-L$, as then the $\g{s}{e}$ would be contained in $|L\otimes \calO_C|=\g{r}{d}$, as observed in~\cite[Lemma~4.1]{Lelli_Chiesa_2015}, which cannot occur as $s>r$.
		
	However, there are other admissible assignments which potentially do occur. An example of this is given by Lelli-Chiesa--Knutsen in \cite[Remark~12]{Lelli_Chiesa_2015}, where the Lazarsfeld--Mukai bundle $E_{C,\g{3}{d}}$ arises as a sum of two Lazarsfeld--Mukai bundles of expected gonality pencils on curves in $|H-L|$ and $|L|$, as in~\cite[Remark~6.18]{auel_haburcak_2022} and \Cref{ex:LM_bund_on_12_2_7}. We give another example below.
	
	\begin{example}\label[example]{ex:Lam100_2_51}
		Suppose we want to find the smallest $e$ such that $\BN{100}{2}{51}\subset \BN{100}{3}{e}$. To find some expectation for $e$, we consider a K3 surface with $\Pic(S)=\Lambda^2_{100,51}$. Checking admissible assignments for a Lazarsfeld--Mukai bundle $E=E_{C,\g{3}{e}}$ of rank $4$ for smooth irreducible $C\in|H|$, we see that there is the admissible assignment $H-L \subset E$, but as noted above this does not come from a terminal filtration. 
		
		The admissible assignment giving the next smallest bound on $c_2(E)$ is of type $2\subset 4$, giving $c_2(E)\geq 77$. In fact, one can use this admissible assignment to construct a $\g{3}{77}$ on $C$. To wit, $E_{H-L,\g{1}{26}}\oplus E_{L,\g{1}{2}}$ is a Lazarsfeld--Mukai bundle of a $\g{3}{77}$ on $C$, and we note that $\BN{100}{3}{77}$ is maximal as $\rho(100,3,77)=-4$. One can readily check that the potential containment \[\BN{100}{2}{51}\stackrel{?}{\subset}\BN{100}{3}{77}\] is not secant expected, though it is K3-expected, by the above, and perhaps the containment is valid.
		
		For $d\geq 52$, there are in fact no containments. By computing admissible assignments, one can readily verify that on a K3 surface with $\Pic(S)=\Lambda^2_{100,d}$ with $d\geq 52$, a smooth irreducible $C\in|H|$ admits no $\g{3}{e}$ with $\rho(g,3,e)<0$. We note that for $d\geq 52$, $(S,H)$ satisfies decomposition rigidity, see~\cite[Definition~2.3]{ahk_2024}, and one can readily check whether $C$ admits a $\g{3}{e}$ using~\cite[Theorem~2.5]{ahk_2024}. For $e\leq 76$, \Cref{thm:rho_k} suffices, as one can also show that $\kappa(100,2,d)\geq 26$ for $d\geq 52$, while $\kappa(100,3,e)\leq 25$ for $e\leq 76$; however, $\kappa(100,2,52)=26=\kappa(100,3,77)$, and one checks the admissible assignments to show a non-containment. We note that for $d\geq 53$, the non-containments follow from~\cite[Theorem~5.1]{ahk_2024} as well.
		
		For lower $d$, curves on K3 surfaces with $\Pic(S)=\Lambda^2_{100,d}$ admit a $\g{3}{e}$ for $e\leq 77$, and Lazarsfeld--Mukai bundles can be constructed from admissible assignments, see \Cref{ex:Lam100_2_smalld} below.
	\end{example}
	
	We give some more detail when $d$ is smaller, as various admissible assignments exist. It appears that the admissible assignments giving the smallest lower bounds on $c_2(E)$ are of the form
	\begin{itemize}
		\item for a terminal filtration of type $1\subset 2 \subset s+1$, with $c_1(E_1)=H-2L$, and $c_1(E_2)=H-L$,
		\item for a terminal filtration type $1\subset s+1$, $H-2L\subset E$,
		\item for a terminal filtration type $2\subset s+1$, $c_1(E_1)=H-L$.
	\end{itemize}
	
	When $d$ is small, the assignment of type $1\subset 2 \subset s+1$ seems to give the smallest bound on $c_2(E)$. As $d$ increases, the bound from the $1\subset 2 \subset s+1$ assignment increases, and the assignment $H-2L\subset E$ gives the smallest bound on $c_2(E)$. As $d$ increases further, the bound obtained from the assignment of type $2\subset s+1$ gives the smallest bound on $c_2(E)$. We give a few examples where the three admissible assignments give different bounds. 
	
	\begin{example}\label[example]{ex:Lam100_2_smalld}
		We work on K3 surfaces with $\Pic(S)=\Lambda^2_{100,d}$, and look for admissible assignments for the Lazarsfeld--Mukai bundle $E=E_{C,\g{3}{e}}$. As \Cref{ex:Lam100_2_51} deals with $d\geq 51$, we focus on $d\leq 50$.
		
		For $d\leq 19$, as $\Delta(100,2,d)>0$ there are no such K3 surfaces. 
		
		For $d=20$, there are over $2000$ admissible assignments, a few which give smaller lower bounds on $c_2(E)$ than the above assignments, and we don't comment on this further.
		
		For $21\leq d \leq 36$, the admissible assignment of type $1\subset 2 \subset 4$ with $c_1(E_1)=H-2L$ and $c_1(E_2)=H-L$ gives the smallest lower bound, and one can check that $(H-2L)\oplus L\oplus E_{L,\g{1}{2}}$ gives a Lazarsfeld--Mukai bundle attaining the bound.
		
		For $d=37$, the admissible assignments above of type $1\subset 4$ and $2\subset 4$ give the same lower bound, $c_2(E)\geq70$, while the admissible assignment of type $1\subset 2 \subset 4$ above gives $c_2(E)\geq 72$. One can check that both $E_{H-L,\g{1}{33}}\oplus E_{L,\g{1}{2}}$ and $(H-2L)\oplus L \oplus E_{L,\g{1}{2}}$ are Lazarsfeld--Mukai bundles of a $\g{3}{70}$ on $C\in|H|$.
		
		For $38\leq d\leq 50$, the admissible assignment of type $2\subset 4$ gives the smallest lower bound on $c_2(E)$, and one obtains a Lazarsfeld--Mukai bundle as a sum of Lazarsfeld--Mukai bundles of expected gonality pencils, namely, $E_{H-L,\g{1}{f}}\oplus E_{L,\g{1}{2}}$, with $f=\floor{\frac{g-d+1+3}{2}}$ (the expected gonality of a smooth curve in $|H-L|$), giving $c_2(E)\geq d+f=52+\floor{\frac{d}{2}}$, as in~\cite[Appendix~A]{Lelli_Chiesa_2015}. We note that in this range, $H$ admits two flexible decompositions, $H=(H-L)+L=(H-2L)+2L$. For a definition of flexible decomposition, see~\cite[Definition~2.1]{ahk_2024}.
	\end{example}
	
	Thus, when $d$ is not too small, one obtains K3-expected containments of the form $\BN{g}{r}{d}\subset \BN{g}{s}{e}$ for $r<s$, and expects non-containments when $e$ is smaller than the bound given by the admissible assignment for a terminal filtration of type $2\subset s+1$ with $c_1(E_1)=H-L$.
	
	\begin{conj}\label[conj]{conj:r_leq_s_gen_conj}
		For $g$ and $d$ sufficiently large, $d,e\leq g-1$, and $2\leq r<s$,
		\begin{itemize}
			\item if $e< d-2r+s+\frac{g-d+r+1}{2}+\frac{(s-2)(r-1)-1}{s-1}$, then $\BN{g}{r}{d}\nsubseteq\BN{g}{s}{e}$,
			\item if $e\geq d-2r+s+\frac{g-d+r+1}{2}+\frac{(s-2)(r-1)-1}{s-1}$, then $\BN{g}{r}{d}\subset \BN{g}{s}{e}$.
		\end{itemize}
	\end{conj}
	
	The range of $g$ and $d$ are ``sufficiently large", as determining the precise range where only a few admissible assignments exist appears involved. Indeed, the existence of admissible assignments is deeply related to the existence of flexible decompositions of $H$, and, as in~\cite{ahk_2024}, the range where is can be shown that $H$ admits only one flexible decomposition can be quite intricate.

	Of course, simply because curves on K3 surfaces seem to agree with this prediction does not mean the (non-)containments hold. Explicit constructions giving the predicted containments appear out of reach in general. Numerically, one can readily verify some non-containments by computing all admissible assignments.

\vfill

\vspace*{3cm}

\begin{thebibliography}{10}
	
	\bibitem{Accola79_castelnuovo_curves}
	Robert D.~M. Accola.
	\newblock On {C}astelnuovo's inequality for algebraic curves. {I}.
	\newblock {\em Trans. Amer. Math. Soc.}, 251:357--373, 1979.
	
	\bibitem{aprodu}
	Marian Aprodu.
	\newblock Lazarsfeld--{M}ukai bundles and applications.
	\newblock In {\em Commutative algebra}, pages 1--23. Springer, New York, 2013.
	
	\bibitem{arbarello_bruno_sernesi}
	Enrico Arbarello, Andrea Bruno, and Edoardo Sernesi.
	\newblock {M}ukai's program for curves on a {K3} surface.
	\newblock {\em Algebraic Geom.}, 1(5):532--557, 2014.
	
	\bibitem{ACGH}
	Enrico Arbarello, Maurizio Cornalba, Phillip~A. Griffiths, and Joe Harris.
	\newblock {\em Geometry of Algebraic Curves}.
	\newblock Springer-Verlag, 1985.
	
	\bibitem{auel_haburcak_2022}
	Asher Auel and Richard Haburcak.
	\newblock {M}aximal {B}rill--{N}oether loci via {K}3 surfaces, 2022.
	\newblock arXiv:2206.04610.
	
	\bibitem{ahk_2024}
	Asher Auel, Richard Haburcak, and Andreas~Leopold Knutsen.
	\newblock Distinguishing {B}rill--{N}oether loci, 2024.
	\newblock arXiv:2406.19993.
	
	\bibitem{ahl_2024}
	Asher Auel, Richard Haburcak, and Hannah Larson.
	\newblock Maximal {B}rill--{N}oether loci via the gonality stratification,
	2024, to appear in Michigan Math J.
	\newblock arXiv:2310.09954.
	
	\bibitem{Ballico_Keem_96}
	Edoardo Ballico and Changho Keem.
	\newblock On linear series on general {$k$}-gonal projective curves.
	\newblock {\em Proc. Amer. Math. Soc.}, 124(1):7--9, 1996.
	
	\bibitem{bh_2024maximal}
	Andrei Bud and Richard Haburcak.
	\newblock Maximal {B}rill--{N}oether loci via degenerations and double covers,
	2024, to appear in Ann. Sc. Norm. Super. Pisa Cl. Sci. (5).
	\newblock arXiv:2404.15066.
	
	\bibitem{Castelnuovo1889}
	Guido Castelnuovo.
	\newblock Una applicazione della geometria enumerativa alle curve algebriche.
	\newblock {\em Rendiconti del Circolo Matematico di Palermo (1884-1940)},
	3(1):27--37, Dec 1889.
	
	\bibitem{Chiantini_Ciliberto_halphen_99}
	Luca Chiantini and Ciro Ciliberto.
	\newblock Towards a {H}alphen theory of linear series on curves.
	\newblock {\em Trans. Amer. Math. Soc.}, 351(6):2197--2212, 1999.
	
	\bibitem{CHOI2022}
	Youngook Choi and Seonja Kim.
	\newblock Linear series on a curve of compact type bridged by a chain of
	elliptic curves.
	\newblock {\em Indagationes Mathematicae}, pages 844--860, 2022.
	
	\bibitem{cook-powell_jensen}
	Kaelin Cook-Powell and David Jensen.
	\newblock Components of {B}rill--{N}oether loci for curves with fixed gonality.
	\newblock {\em Michigan Math. J.}, 71(1):19--45, 2022.
	
	\bibitem{Coppens_91}
	Marc Coppens.
	\newblock The gonality of general smooth curves with a prescribed plane nodal
	model.
	\newblock {\em Math. Ann.}, 289(1):89--93, 1991.
	
	\bibitem{Copopens_Kato_gon_nodal_plane_curves}
	Marc Coppens and Takao Kato.
	\newblock The gonality of smooth curves with plane models.
	\newblock {\em Manuscripta Math.}, 70(1):5--25, 1990.
	
	\bibitem{Coppens_Kato_nontriv_plane_94}
	Marc Coppens and Takao Kato.
	\newblock Nontrivial linear systems on smooth plane curves.
	\newblock {\em Math. Nachr.}, 166:71--82, 1994.
	
	\bibitem{Coppens_Martens_99}
	Marc Coppens and Gerriet Martens.
	\newblock Linear series on a general {$k$}-gonal curve.
	\newblock {\em Abh. Math. Sem. Univ. Hamburg}, 69:347--371, 1999.
	
	\bibitem{Coppens_Martens_4gonal}
	Marc Coppens and Gerriet Martens.
	\newblock Linear series on 4-gonal curves.
	\newblock {\em Math. Nachr.}, 213:35--55, 2000.
	
	\bibitem{Coppens_Martens_4_gonal_vample}
	Marc Coppens and Gerriet Martens.
	\newblock Very ample linear series on quadrigonal curves.
	\newblock {\em Bull. Belg. Math. Soc. Simon Stevin}, 29(4):423--434, 2022.
	
	\bibitem{EdidinThesis}
	Dan Edidin.
	\newblock {B}rill--{N}oether theory in codimension-two.
	\newblock {\em J. Algebraic Geom.}, 2(1):25--67, 1993.
	
	\bibitem{eisenbud_harris}
	David Eisenbud and Joe Harris.
	\newblock The {K}odaira dimension of the moduli space of curves of genus
	{$\ge$} 23.
	\newblock {\em Inventiones mathematicae}, 90(2):359--387, Jun 1987.
	
	\bibitem{Eisenbud_Harris_1989}
	David Eisenbud and Joe Harris.
	\newblock Irreducibility of some families of linear series with
	{B}rill--{N}oether number {$ -1 $}.
	\newblock {\em Annales scientifiques de l'\'Ecole Normale Sup\'erieure}, Ser.
	4, 22(1):33--53, 1989.
	
	\bibitem{Farkas2000}
	Gavril Farkas.
	\newblock The geometry of the moduli space of curves of genus 23.
	\newblock {\em Mathematische Annalen}, 318(1):43--65, Sep 2000.
	
	\bibitem{Farkas_2001}
	Gavril Farkas.
	\newblock {B}rill--{N}oether loci and the gonality stratification of
	{$\mathcal{M}_g$}.
	\newblock {\em {J}. Reine. Angew. {M}ath.}, 2001(539):185--200, 2001.
	
	\bibitem{Farkas_2008}
	Gavril Farkas.
	\newblock Higher ramification and varieties of secant divisors on the generic
	curve.
	\newblock {\em J. Lond. Math. Soc. (2)}, 78(2):418--440, 2008.
	
	\bibitem{farkas_feyz_Rojas_2025}
	Gavril Farkas, Soheyla Feyzbakhsh, and Andrés Rojas.
	\newblock {H}urwitz--{B}rill--{N}oether theory via {K}3 surfaces and stability
	conditions, 2025.
	\newblock arXiv:2505.19890.
	
	\bibitem{gieseker}
	David Gieseker.
	\newblock Stable curves and special divisors: {P}etri's conjecture.
	\newblock {\em Invent. Math.}, 66(2):251--275, 1982.
	
	\bibitem{griffiths_harris}
	Phillip Griffiths and Joseph Harris.
	\newblock On the variety of special linear systems on a general algebraic
	curve.
	\newblock {\em Duke Math. J.}, 47(1):233--272, 1980.
	
	\bibitem{Haburcak_BNspecial_k3_2024}
	Richard Haburcak.
	\newblock Curves on {B}rill-{N}oether special {K}3 surfaces.
	\newblock {\em Math. Nachr.}, 297(12):4497--4509, 2024.
	
	\bibitem{h_tib_theta}
	Richard Haburcak and Montserrat~Teixidor i~Bigas.
	\newblock Some reducible and irreducible {B}rill--{N}oether loci, 2025.
	\newblock arXiv:2503.16255.
	
	\bibitem{harris_mumford}
	Joe Harris and David Mumford.
	\newblock On the {K}odaira dimension of the moduli space of curves.
	\newblock {\em Invent. Math.}, 67(1):23--86, Feb 1982.
	
	\bibitem{Hartshorne_AG}
	Robin Hartshorne.
	\newblock {\em Algebraic geometry}, volume No. 52 of {\em Graduate Texts in
		Mathematics}.
	\newblock Springer-Verlag, New York-Heidelberg, 1977.
	
	\bibitem{huybrechts_lehn}
	Daniel Huybrechts and Manfred Lehn.
	\newblock {\em The geometry of moduli spaces of sheaves}.
	\newblock Cambridge Mathematical Library. Cambridge University Press,
	Cambridge, second edition, 2010.
	
	\bibitem{jensen_payne_2021recent}
	David Jensen and Sam Payne.
	\newblock Recent developments in {B}rill--{N}oether theory, 2021.
	\newblock To appear in EMS Volume for the BMS Thematic Einstein Semester on
	Algebraic Geometry, arXiv:2111.00351.
	
	\bibitem{jensen2020brillnoether}
	David Jensen and Dhruv Ranganathan.
	\newblock {B}rill--{N}oether theory for curves of a fixed gonality.
	\newblock {\em Forum of Mathematics, Pi}, 9:e1, 2021.
	
	\bibitem{jensen_ranganathan}
	David Jensen and Dhruv Ranganathan.
	\newblock {B}rill--{N}oether theory for curves of a fixed gonality.
	\newblock {\em Forum Math. Pi}, 9:Paper No. e1, 33, 2021.
	
	\bibitem{Knutsen_Lelli_Chiesa_Mongardi_BN_abelian_surf}
	Andreas~Leopold Knutsen, Margherita Lelli-Chiesa, and Giovanni Mongardi.
	\newblock Severi varieties and {B}rill-{N}oether theory of curves on abelian
	surfaces.
	\newblock {\em J. Reine Angew. Math.}, 749:161--200, 2019.
	
	\bibitem{Lange_moduli_curves_rat_maps}
	Herbert Lange.
	\newblock Moduli spaces of algebraic curves with rational maps.
	\newblock {\em Math. Proc. Cambridge Philos. Soc.}, 78(2):283--292, 1975.
	
	\bibitem{larson_larson_vogt_2020global}
	Eric Larson, Hannah Larson, and Isabel Vogt.
	\newblock Global {B}rill--{N}oether theory over the {H}urwitz space.
	\newblock {\em Geom. Topol.}, 29(1):193--257, 2025.
	
	\bibitem{larson_refined_BN_Hurwitz}
	Hannah Larson.
	\newblock A refined {B}rill--{N}oether theory over {H}urwitz spaces.
	\newblock {\em Invent. Math.}, 224(3):767--790, 2021.
	
	\bibitem{larson_2024brillnoethertheorysmoothcurves}
	Hannah Larson and Sameera Vemulapalli.
	\newblock {B}rill--{N}oether theory of smooth curves in the plane and on
	{H}irzebruch surfaces, 2024.
	\newblock arXiv:2408.12678.
	
	\bibitem{Larson_trigonal}
	Hannah~K. Larson.
	\newblock Refined {B}rill-{N}oether theory for all trigonal curves.
	\newblock {\em Eur. J. Math.}, 7(4):1524--1536, 2021.
	
	\bibitem{lazarsfeld:Brill-Noether_without_degenerations}
	Robert Lazarsfeld.
	\newblock {B}rill--{N}oether--{P}etri without degenerations.
	\newblock {\em J. Differ. Geom.}, 23(3):299--307, 1986.
	
	\bibitem{Laz_lect_lin_ser}
	Robert Lazarsfeld.
	\newblock Lectures on linear series.
	\newblock In {\em Complex algebraic geometry ({P}ark {C}ity, {UT}, 1993)},
	volume~3 of {\em IAS/Park City Math. Ser.}, pages 161--219. Amer. Math. Soc.,
	Providence, RI, 1997.
	\newblock With the assistance of Guillermo Fern\'andez del Busto.
	
	\bibitem{Lelli-Chiesa_the_gieseker_petri_divisor_g_le_13}
	Margherita Lelli-Chiesa.
	\newblock The {G}ieseker--{P}etri divisor in {$M_g$} for {$g\leq 13$}.
	\newblock {\em Geom. Dedicata}, 158:149--165, 2012.
	
	\bibitem{Lelli_Chiesa_2013}
	Margherita Lelli-Chiesa.
	\newblock Stability of rank-3 {L}azarsfeld--{M}ukai bundles on {$K3$} surfaces.
	\newblock {\em Proc. Lond. Math. Soc. (3)}, 107(2):451--479, 2013.
	
	\bibitem{Lelli_Chiesa_2015}
	Margherita Lelli-Chiesa.
	\newblock Generalized {L}azarsfeld--{M}ukai bundles and a conjecture of
	{D}onagi and {M}orrison.
	\newblock {\em Adv. Math.}, 268:529--563, 2015.
	
	\bibitem{Maroni_46_trigonal}
	Arturo Maroni.
	\newblock Le serie lineari speciali sulle curve trigonali.
	\newblock {\em Ann. Mat. Pura Appl. (4)}, 25:343--354, 1946.
	
	\bibitem{Martens_96}
	Gerriet Martens.
	\newblock On curves of odd gonality.
	\newblock {\em Arch. Math. (Basel)}, 67(1):80--88, 1996.
	
	\bibitem{Martens_Cliff_1}
	Henrik~H. Martens.
	\newblock Varieties of special divisors on a curve. {II}.
	\newblock {\em J. Reine Angew. Math.}, 233:89--100, 1968.
	
	\bibitem{Mori_complete_int}
	Shigefumi Mori.
	\newblock On degrees and genera of curves on smooth quartic surfaces in {${\bf
			P}\sp 3$}.
	\newblock {\em Nagoya Math. J.}, 96:127--132, 1984.
	
	\bibitem{Morrison_k3}
	David~R. Morrison.
	\newblock On {$K3$}\ surfaces with large {P}icard number.
	\newblock {\em Invent. Math.}, 75(1):105--121, 1984.
	
	\bibitem{Nikulin_79}
	Viacheslav~V. Nikulin.
	\newblock Integral symmetric bilinear forms and some of their applications.
	\newblock {\em Math. USSR-Izv.}, 14(1):103--167, 1980.
	
	\bibitem{OGrady_irreducible_NL_divisors}
	Kieran~G. O'Grady.
	\newblock {\em Moduli of {A}belian and {K3} {S}urfaces}.
	\newblock {PhD} thesis, Brown University, 1986.
	
	\bibitem{pflueger}
	Nathan Pflueger.
	\newblock {B}rill--{N}oether varieties of {$k$}-gonal curves.
	\newblock {\em Adv. Math.}, 312:46--63, 2017.
	
	\bibitem{pflueger_legos}
	Nathan Pflueger.
	\newblock Linear series with {$\rho<0$} via thrifty {L}ego building.
	\newblock {\em J. Reine Angew. Math.}, 797:193--228, 2023.
	
	\bibitem{Sernesi_1984}
	Edoardo Sernesi.
	\newblock On the existence of certain families of curves.
	\newblock {\em Invent. Math.}, 75(1):25--57, 1984.
	
	\bibitem{steffen_1998}
	Frauke Steffen.
	\newblock A generalized principal ideal theorem with an application to
	{B}rill--{N}oether theory.
	\newblock {\em Invent. Math.}, 132(1):73--89, 1998.
	
	\bibitem{bigas2023brillnoether}
	Montserrat Teixidor~i Bigas.
	\newblock {B}rill--{N}oether loci.
	\newblock {\em Manuscripta Math.}, 176(1):Paper No. 14, 17, 2025.
	
\end{thebibliography}
\end{document}